\documentclass[a4paper]{amsart}

\usepackage[utf8]{inputenc}
\usepackage[T1]{fontenc}
\usepackage[english]{babel}

\usepackage{amsfonts}
\usepackage{amsmath}
\usepackage{amsrefs}
\usepackage{amssymb}
\usepackage{amstext}
\usepackage{amsthm}

\usepackage{fullpage}
\usepackage{hyperref}
\usepackage{mathrsfs}
\usepackage{mathtools}
\usepackage{multicol}

\usepackage{tikz}
\usetikzlibrary{backgrounds,shapes}

\newcommand{\qbinom}[2]{\genfrac{[}{]}{0pt}{}{#1}{#2}}

\newcommand*{\mystrut}{\rule[-0.4\baselineskip]{0pt}{1.2\baselineskip}}
\renewcommand*{\boxed}[1]{\framebox{\mystrut $#1$}}
\newcommand{\domino}[2]{\boxed{\genfrac{}{}{0pt}{}{#1}{#2}}}

\newcommand\multiset[2]%
{\mathchoice{\left(\kern-0.4em{\binom{#1}{#2}}\kern-0.4em\right)}
	{\bigl(\kern-0.2em{\binom{#1}{#2}}\kern-0.2em\bigr)}
	{\bigl(\kern-0.2em{\binom{#1}{#2}}\kern-0.2em\bigr)}
	{\bigl(\kern-0.2em{\binom{#1}{#2}}\kern-0.2em\bigr)}}

\let\existstemp\exists \renewcommand*{\exists}{\mathop \existstemp}
\let\foralltemp\forall \renewcommand*{\forall}{\mathop \foralltemp}

\def\quotient#1#2{\raise1ex\hbox{$#1$}\Big/\lower1ex\hbox{$#2$}}

\newcommand{\<}{\langle}
\renewcommand{\>}{\rangle}

\theoremstyle{plain}

\theoremstyle{definition}
\newtheorem{theorem}{Theorem}[section]

\newtheorem{conjecture}[theorem]{Conjecture}
\newtheorem{corollary}[theorem]{Corollary}
\newtheorem{definition}[theorem]{Definition}
\newtheorem{lemma}[theorem]{Lemma}
\newtheorem{proposition}[theorem]{Proposition}

\theoremstyle{remark}

\makeatletter%
\renewenvironment{proof}[1][\proofname]{%
	\par\pushQED{\qed}\normalfont%
	\topsep6\p@\@plus6\p@\relax
	\trivlist\item[\hskip\labelsep\bfseries#1\@addpunct{.}]%
	\ignorespaces
}{%
	\qedhere 
}
\makeatother

\newcommand{\area}{\mathsf{area}}
\newcommand{\bounce}{\mathsf{bounce}}
\newcommand{\dinv}{\mathsf{dinv}}
\newcommand{\pmaj}{\mathsf{pmaj}}
\newcommand{\uarea}{\protect\underline{\mathsf{area}}}
\newcommand{\ubounce}{\protect\underline{\mathsf{bounce}}}

\newcommand{\PF}{\mathsf{PF}} 
\newcommand{\PDP}{\mathsf{PDP}} 
\newcommand{\PP}{\mathsf{PP}} 
\newcommand{\RP}{\mathsf{RP}} 
\newcommand{\LP}{\mathsf{LP}} 

\title{Parallelogram polyominoes, partially labelled \\ ~Dyck paths, and the Delta conjecture}

\author{Michele D'Adderio}
\address{Universit\'e Libre de Bruxelles (ULB)\\D\'epartement de Math\'ematique\\ Boulevard du Triomphe, B-1050 Bruxelles\\ Belgium}\email{mdadderi@ulb.ac.be}

\author{Alessandro Iraci}
\address{Universit\'a di Pisa and Universit\'e Libre de Bruxelles (ULB)\\Dipartimento di Matematica\\ Largo Bruno Pontecorvo 5, 56127 Pisa\\ Italia}\email{iraci@student.dm.unipi.it}

\begin{document}

\begin{abstract}
	We introduce $\area$, $\bounce$ and $\dinv$ statistics on decorated parallelogram polyominoes, and prove that some of their $q, t$-enumerators match $\langle \Delta_{h_m} e_{n+1},s_{k+1,1^{n-k}}\rangle$, extending in this way the work in (Aval et al. 2014). Also, we provide a bijective connection between decorated parallelogram polyominoes and decorated labelled Dyck paths, which allows us to prove the combinatorial interpretation of the coefficient $\langle \Delta_{e_{m+n-k-1}}'e_{m+n},h_m h_n\rangle$ predicted by the Delta conjecture in (Haglund et al. 2015). Finally, we define a statistic $\pmaj$ on partially labelled Dyck paths, which provides another conjectural combinatorial interpretation of $\Delta_{h_{\ell}}\Delta_{e_{n-k-1}}'e_n$, cf. (Haglund et al. 2015). 
	
	\noindent This is the full version of (D'Adderio, Iraci 2017) arXiv:1711.03923.
\end{abstract}
	
\maketitle
\tableofcontents

\section*{Introduction}

In \cite{Aval-DAdderio-Dukes-Hicks-LeBorgne-2014}, statistics $\area$, $\bounce$ and $\dinv$ have been defined on parallelogram polyominoes, and some of their $q,t$-enumerators have been shown to be $\< \Delta_{e_{m+n}} e_{m+n}, h_m h_n \> = \< \Delta_{h_m} e_{n+1}, s_{1^{n+1}} \>$. In particular, the first author proposed a combinatorial interpretation of the full symmetric function $\Delta_{h_m} e_{n+1}$ at $q=1$ in terms of labelled parallelogram polyominoes, and asked for a $\dinv$ statistic that could give the more general $\Delta_{h_m} e_{n+1}$ (cf. \cite[Equations (8.1) and (8.14)]{Aval-Bergeron-Garsia-2015}). 

In \cite{Haglund-Remmel-Wilson-2015} the authors state the so called Delta conjecture, which predicts a combinatorial interpretation of the symmetric function $\Delta_{e_k} e_n$. This is a generalization of the \emph{Shuffle conjecture} stated in \cite{HHLRU-2005} and \cite{Haglund-Morse-Zabrocki-2012} and proved in \cite{Carlsson-Mellit-ShuffleConj-2015}, which is related to the famous diagonal harmonics discovered by Garsia and Haiman in their work towards a proof of the Schur positivity of Macdonald polynomials (cf. \cite{Garsia-Haiman-PNAS-1993,Garsia-Haiman-qLagrange-1996,Haiman-nfactorial-2001,Haiman-Vanishing-2002}). 

Other than the results mentioned in \cite{Haglund-Remmel-Wilson-2015}, other consequences of this conjecture have been proved, in particular in \cite{DAdderio-VandenWyngaerd-2017,Garsia-Haglund-Remmel-Yoo-2017,Romero-Deltaq1-2017,Zabrocki-4Catalan-2016}, while the general Delta conjecture remains open.

The delta operator $\Delta_f$ has been defined for any symmetric function $f$ by Bergeron, Garsia, Haiman and Tesler, and in fact in \cite{Haglund-Remmel-Wilson-2015} the authors provide a generalization of their Delta conjecture for the symmetric function $\Delta_{h_{\ell}e_k}e_n$ in terms of partially labelled Dyck paths.

In this work we extend the results in \cite{Aval-DAdderio-Dukes-Hicks-LeBorgne-2014}, by providing a combinatorial interpretation of the more general $\< \Delta_{h_m} e_{n+1}, s_{k+1,1^{n-k}} \>$ in terms of decorated parallelogram polyominoes.

Also, we prove the formula for $\< \Delta_{e_{m+n-k-1}}' e_{m+n}, h_m h_n \>$ predicted by the Delta conjecture in \cite{Haglund-Remmel-Wilson-2015}, by providing a recursion of these polynomials. This, together with the results in \cite{DAdderio-VandenWyngaerd-2017}, completely solves Problem 8.1 in \cite{Haglund-Remmel-Wilson-2015}.

Surprisingly, these two solutions are intimately related: indeed, with a bijection, we show that these two combinatorial polynomials actually coincide.

In order to prove these results, we provide some symmetric function identities. The proofs of these are based on theorems in \cite{DAdderio-VandenWyngaerd-2017} and \cite{Haglund-Schroeder-2004}.

Finally, we introduce a $\pmaj$ statistic on partially labelled Dyck paths, providing another combinatorial interpretation of $\Delta_{h_{\ell} e_k}e_n$. Then, for the special case $k=0$, we describe a bijection with labelled parallelogram polyominoes, which allows us to define both a $\dinv$ and a $\pmaj$ statistic on these objects: this answers the question in \cite[Equation (8.14)]{Aval-Bergeron-Garsia-2015}.

The paper is organized as follows: in Section 1 we review some basic notation and results from the theory of symmetric functions and Macdonald's polynomials. In Sections~2 to 4 we introduced some of the combinatorial objects we are going to deal with, and some relevant statistics on them, by developing some of their theory. In Section~5 we prove a few identities of symmetric functions that we need to get our main results, that are presented and proved in Section~6. In Sections~7 and 8 we introduce some more combinatorial objects and we state our general conjecture involving the statistic $\pmaj$.

\section{Symmetric function notation}
The main references that we will use for symmetric functions
are \cite{Macdonald-Book-1995} and \cite{Stanley-Book-1999}. This section is very similar to \cite[Section~1]{DAdderio-VandenWyngaerd-2017}, but we need it to fix the notations and state the identities we use later in this work.

The standard bases of the symmetric functions that will appear in our
calculations are the complete homogeneous $\{h_{\lambda}\}_{\lambda}$, elementary $\{e_{\lambda}\}_{\lambda}$, power $\{p_{\lambda}\}_{\lambda}$ and Schur $\{s_{\lambda}\}_{\lambda}$ bases.

\emph{We will use implicitly the usual convention that $e_0 = h_0 = 1$ and $e_k = h_k = 0$ for $k < 0$.}

The ring $\Lambda$ of symmetric functions can be thought of as the polynomial ring in the power sum generators $p_1, p_2, p_3,\dots$. This ring has a grading $\Lambda=\bigoplus_{n\geq 0}\Lambda^{(k)}$ given by assigning degree $i$ to $p_i$ for all $i\geq 1$. As we are working with Macdonald symmetric functions involving two parameters $q$ and $t$, we will consider this polynomial ring over the field $\mathbb{Q}(q,t)$. We will make extensive use of the \emph{plethystic notation}.

With this notation we will be able to add and subtract alphabets, which will be represented as sums of monomials $X = x_1 + x_2 + x_3+\cdots $. Then, given a symmetric function $f$, and thinking of it as an element of $\Lambda$, we denote by $f[X]$ the expression $f$ with $p_k$ replaced by $x_{1}^{k}+x_{2}^{k}+x_{3}^{k}+\cdots$, for all $k$. 

We have for example the addition formulas
\begin{align}
	p_k[X+Y]=p_k[X]+p_k[Y]\quad \text{ and } \quad p_k[X-Y]=p_k[X]-p_k[Y],
\end{align}
and
\begin{align}
	\label{eq:e_h_sum_alphabets}
	e_n[X+Y]=\sum_{i=0}^ne_{n-i}[X]e_i[Y]\quad \text{ and } \quad  h_n[X+Y]=\sum_{i=0}^nh_{n-i}[X]h_i[Y].
\end{align}
Notice in particular that $p_k[-X]$ equals $-p_k[X]$ and not $(-1)^kp_k[X]$. As the latter sort of negative sign can be also useful, it is customary to use the notation $\epsilon$ to express it: we will have $p_k[\epsilon X] = (-1)^k p_k[X]$, so that, in general,
\begin{align}
	\label{eq:minusepsilon}
	f[-\epsilon X] = \omega f[X]
\end{align}
for any symmetric function $f$, where $\omega$ is the fundamental algebraic involution which sends $e_k$ to $h_k$, $s_{\lambda}$ to $s_{\lambda'}$ and $p_k$ to $(-1)^kp_k$.

We denote by $\<\, , \>$ the \emph{Hall scalar product} on symmetric functions, which can be defined by saying that the Schur functions form an orthonormal basis. With this definition, we have the orthogonality
\begin{align}
	\< p_{\lambda}, p_{\mu} \> = z_{\mu} \chi(\lambda=\mu)
\end{align}
which defines the integers $z_{\mu}$, where $\chi(\mathcal{P})=1$ if the statement $\mathcal{P}$ is true, and $\chi(\mathcal{P})=0$ otherwise.

Recall also the \emph{Cauchy identities}
\begin{align}
	\label{eq:Cauchy_identities}
	e_n[XY] = \sum_{\lambda\vdash n} s_{\lambda}[X] s_{\lambda'}[Y] \quad \text{ and } \quad h_n[XY] = \sum_{\lambda\vdash n} s_{\lambda}[X] s_{\lambda}[Y].
\end{align}
With the symbol ``$\perp$'' we denote the operation of taking the adjoint of an operator with respect to the Hall scalar product, i.e.
\begin{align}
	\< f^\perp g, h \> = \< g, fh \> \quad \text{ for all } f, g, h \in \Lambda.
\end{align}
We introduce also the operator
\begin{align}
	\tau_z f[X] \coloneqq f[X+z] \qquad \text{ for all } f[X] \in \Lambda,
\end{align}
so for example
\begin{align}
	\tau_{-\epsilon} f[X] = f[X-\epsilon] \qquad \text{ for all } f[X] \in \Lambda.
\end{align}
We refer also to \cite{Haglund-Book-2008} for more informations on this topic.

\subsection{Macdonald symmetric function toolkit}

For a partition $\mu\vdash n$, we denote by
\begin{align}
	\widetilde{H}_{\mu} \coloneqq \widetilde{H}_{\mu}[X]=\widetilde{H}_{\mu}[X;q,t]=\sum_{\lambda\vdash n}\widetilde{K}_{\lambda \mu}(q,t)s_{\lambda}
\end{align}
the \emph{(modified) Macdonald polynomials}, where
\begin{align}
	\widetilde{K}_{\lambda \mu} \coloneqq \widetilde{K}_{\lambda \mu}(q,t)=K_{\lambda \mu}(q,1/t)t^{n(\mu)}\quad \text{ with }\quad n(\mu)=\sum_{i\geq 1}\mu_i(i-1)
\end{align}
are the \emph{(modified) Kostka coefficients} (see \cite[Chapter~2]{Haglund-Book-2008} for more details). 

The set $\{\widetilde{H}_{\mu}[X;q,t]\}_{\mu}$ is a basis of the ring of symmetric functions $\Lambda$ with coefficients in $\mathbb{Q}(q,t)$. This is a modification of the basis introduced by Macdonald \cite{Macdonald-Book-1995}, and they are the Frobenius characteristic of the so called Garsia-Haiman modules (see \cite{Garsia-Haiman-PNAS-1993}).

If we identify the partition $\mu$ with its Ferrers diagram, i.e. with the collection of cells $\{(i,j)\mid 1\leq i\leq \mu_i, 1\leq j\leq \ell(\mu)\}$, then for each cell $c\in \mu$ we refer to the \emph{arm}, \emph{leg}, \emph{co-arm} and \emph{co-leg} (denoted respectively as $a_\mu(c), l_\mu(c), a_\mu(c)', l_\mu(c)'$) as the number of cells in $\mu$ that are strictly to the right, above, to the left and below $c$ in $\mu$, respectively (see Figure~\ref{fig:notation}).

\begin{figure}[h]
	\centering
	\begin{tikzpicture}[scale=0.4]
		\draw[gray,opacity=.4](0,0) grid (15,10);
		\fill[white] (1,10)|-(3,9)|- (5,7)|-(9,5)|-(13,2)--(15.2,2)|-(1,10.2);
		\draw[gray]  (1,10)|-(3,9)|- (5,7)|-(9,5)|-(13,2)--(15,2)--(15,0)-|(0,10)--(1,10);
		\fill[blue, opacity=.2] (0,3) rectangle (9,4) (4,0) rectangle (5,7); 
		\fill[blue, opacity=.5] (4,3) rectangle (5,4);
		\draw (7,4.5) node {\tiny{Arm}} (3.25,5.5) node {\tiny{Leg}} (6.25, 1.5) node {\tiny{Co-leg}} (2,2.5) node {\tiny{Co-arm}} ;
	\end{tikzpicture}
	\caption{}
	\label{fig:notation}
\end{figure}

We set $M \coloneqq (1-q)(1-t)$ and we define for every partition $\mu$
\begin{align}
	B_{\mu} & \coloneqq B_{\mu}(q,t)=\sum_{c\in \mu}q^{a_{\mu}'(c)}t^{l_{\mu}'(c)} \\
	D_{\mu} & \coloneqq MB_{\mu}(q,t)-1 \\
	T_{\mu} & \coloneqq T_{\mu}(q,t)=\prod_{c\in \mu}q^{a_{\mu}'(c)}t^{l_{\mu}'(c)} \\
	\Pi_{\mu} & \coloneqq \Pi_{\mu}(q,t)=\prod_{c\in \mu}(1-q^{a_{\mu}'(c)}t^{l_{\mu}'(c)}) \\
	w_{\mu} & \coloneqq w_{\mu}(q,t)=\prod_{c\in \mu} (q^{a_{\mu}(c)} - t^{l_{\mu}(c) + 1}) (t^{l_{\mu}(c)} - q^{a_{\mu}(c) + 1}).
\end{align}
Notice that
\begin{align}
	\label{eq:Bmu_Tmu}
	B_{\mu} = e_1[B_{\mu}]\quad \text{ and } \quad T_{\mu}=e_{|\mu|}[B_{\mu}].
\end{align}

It is useful to introduce the so called \emph{star scalar product} on $\Lambda$ given by
\[ \< p_{\lambda},p_{\mu} \>_*=(-1)^{|\mu|-|\lambda|}\prod_{i=1}^{\ell(\mu)}(1-q^{\mu_i})(1-t^{\mu_i}) z_{\mu}\chi(\mu=\lambda). \]

For every symmetric function $f[X]$ and $g[X]$ we have (see \cite[Proposition~1.8]{Garsia-Haiman-Tesler-Explicit-1999})
\begin{align}
	\< f,g\>_*= \< \omega \phi f,g\>=\< \phi \omega f,g\>
\end{align}
where 
\begin{align}
	\phi f[X] \coloneqq f[MX]\qquad \text{ for all } f[X] \in \Lambda.
\end{align}
It turns out that the Macdonald polynomials are orthogonal with respect to the star scalar product: more precisely
\begin{align}
	\label{eq:H_orthogonality}
	\< \widetilde{H}_{\lambda},\widetilde{H}_{\mu}\>_*=w_{\mu}(q,t)\chi(\lambda=\mu).
\end{align}
These orthogonality relations give the following Cauchy identities
\begin{align}
	\label{eq:Mac_Cauchy}
	e_n \left[ \frac{XY}{M} \right] = \sum_{\mu \vdash n} \frac{ \widetilde{H}_{\mu} [X] \widetilde{H}_\mu [Y]}{w_\mu} \quad \text{ for all } n.
\end{align}
The following linear operators were introduced in \cite{Bergeron-Garsia-ScienceFiction-1999,Bergeron-Garsia-Haiman-Tesler-Positivity-1999}, and they are at the basis of the conjectures relating symmetric function coeffcients and $q,t$-combinatorics in this area. 

We define the \emph{nabla} operator on $\Lambda$ by
\begin{align}
	\nabla \widetilde{H}_{\mu} \coloneqq T_{\mu} \widetilde{H}_{\mu} \quad \text{ for all } \mu,
\end{align}
and we define the \emph{Delta} operators $\Delta_f$ and $\Delta_f'$ on $\Lambda$ by
\begin{align}
	\Delta_f \widetilde{H}_{\mu} \coloneqq f[B_{\mu}(q,t)] \widetilde{H}_{\mu} \quad \text{ and } \quad 
	\Delta_f' \widetilde{H}_{\mu}  \coloneqq f[B_{\mu}(q,t)-1] \widetilde{H}_{\mu}, \quad \text{ for all } \mu.
\end{align}
Observe that on the vector space of symmetric functions homogeneous of degree $n$, denoted by $\Lambda^{(n)}$, the operator $\nabla$ equals $\Delta_{e_n}$. Moreover, for every $1\leq k\leq n$,
\begin{align}
	\label{eq:deltaprime}
	\Delta_{e_k} = \Delta_{e_k}' + \Delta_{e_{k-1}}' \quad \text{ on } \Lambda^{(n)},
\end{align}
and for any $k > n$, $\Delta_{e_k} = \Delta_{e_{k-1}}' = 0$ on $\Lambda^{(n)}$, so that $\Delta_{e_n}=\Delta_{e_{n-1}}'$ on $\Lambda^{(n)}$.

We will use the following form of \emph{Macdonald-Koornwinder reciprocity} (see \cite{Macdonald-Book-1995}*{p.~332} or \cite{Garsia-Haiman-Tesler-Explicit-1999}): for all partitions $\alpha$ and $\beta$
\begin{align}
	\label{eq:Macdonald_reciprocity}
	\frac{\widetilde{H}_{\alpha}[MB_{\beta}]}{\Pi_{\alpha}} = \frac{\widetilde{H}_{\beta}[MB_{\alpha}]}{\Pi_{\beta}}.
\end{align}
One of the most important identities in this theory is the following one (see \cite[Theorem~I.2]{Garsia-Haiman-Tesler-Explicit-1999}): for every symmetric function $f[X]$ and every partition $\mu$, we have
\begin{align}
	\label{eq:glenn_formula}
	\< f[X], \widetilde{H}_\mu[X+1]\>_*= \left.\nabla^{-1}\tau_{-\epsilon} f[X]\right|_{X=D_\mu}.
\end{align}

\subsection{Pieri rules and summation formulae}

For a given $k\geq 1$, we define the Pieri coefficients $c_{\mu \nu}^{(k)}$ and $d_{\mu \nu}^{(k)}$ by setting
\begin{align}
	\label{eq:def_cmunu} h_{k}^\perp \widetilde{H}_{\mu}[X] & =\sum_{\nu \subset_k \mu} c_{\mu \nu}^{(k)} \widetilde{H}_{\nu}[X], \\
	\label{eq:def_dmunu} e_{k}\left[\frac{X}{M}\right] \widetilde{H}_{\nu}[X] & = \sum_{\mu \supset_k \nu} d_{\mu \nu}^{(k)} \widetilde{H}_{\mu}[X].
\end{align}
The following identity is \cite[Proposition~5]{Bergeron-Haiman-2013}, written in the notation of \cite{Garsia-Haglund-Xin-Zabrocki-Pieri-2016}, which is coherent with ours:
\begin{align}
	\label{eq:cmunu_recursion}
	c_{\mu \nu}^{(k+1)} = \frac{1}{B_{\mu/\nu}} \sum_{\nu\subset_1 \alpha\subset_k\mu} c_{\mu \alpha}^{(k)} c_{\alpha \nu}^{(1)} \frac{T_{\alpha}}{T_{\mu}} \quad \text{ with } \quad B_{\mu/\nu} \coloneqq B_{\mu} - B_{\nu},
\end{align}
where $\nu\subset_k \mu$ means that $\nu$ is contained in $\mu$ (as Ferrers diagrams) and $\mu/\nu$ has $k$ lattice cells, and the symbol $\mu \supset_k \nu$ is analogously defined. It follows from \eqref{eq:H_orthogonality} that
\begin{align}
	\label{eq:rel_cmunu_dmunu}
	c_{\mu \nu}^{(k)} = \frac{w_{\mu}}{w_{\nu}}d_{\mu \nu}^{(k)}.
\end{align}
For every $m\in \mathbb{N}$ with $m\geq 1$, and for every $\gamma\vdash m$, we have the well-known summation formula (see for example \cite[Equation~(1.35)]{DAdderio-VandenWyngaerd-2017})
\begin{align}
	\label{eq:Pieri_sum1}
	B_{\gamma} = \sum_{\delta\subset_1 \gamma}c_{\gamma \delta}^{(1)}.
\end{align}

\subsection{$q$-notation}

We recall here some standard notations for $q$-analogues. For $n, k\in \mathbb{N}$, we set
\begin{align}
	[0]_q \coloneqq 0, \quad \text{ and } \quad [n]_q \coloneqq \frac{1-q^n}{1-q} = 1+q+q^2+\cdots+q^{n-1} \quad \text{ for } n \geq 1,
\end{align}
\begin{align}
	[0]_q! \coloneqq 1 \quad \text{ and }\quad [n]_q! \coloneqq [n]_q[n-1]_q \cdots [2]_q [1]_q \quad \text{ for } n \geq 1,
\end{align}
and
\begin{align}
	\qbinom{n}{k}_q  \coloneqq \frac{[n]_q!}{[k]_q![n-k]_q!} \quad \text{ for } n \geq k \geq 0, \quad \text{ and } \quad \qbinom{n}{k}_q \coloneqq 0 \quad \text{ for } n < k.
\end{align}
Recall the well-known recursion
\begin{align}
	\label{eq:qbin_recursion}
	\qbinom{n}{k}_q = q^k \qbinom{n-1}{k}_q + \qbinom{n-1}{k-1}_q = \qbinom{n-1}{k}_q + q^{n-k} \qbinom{n-1}{k-1}_q.
\end{align}

Recall also the standard notation for the $q$-\emph{rising factorial}
\begin{align}
	(a;q)_s \coloneqq (1 - a)(1 - qa)(1 - q^2 a) \cdots (1 - q^{s-1} a).
\end{align}
It is well-known (cf. \cite[Theorem~7.21.2]{Stanley-Book-1999}) that (recall that $h_0 = 1$)
\begin{align}
	\label{eq:h_q_binomial}
	h_k[[n]_q] = \frac{(q^{n};q)_k}{(q;q)_k} = \qbinom{n+k-1}{k}_q \quad \text{ for } n \geq 1 \text{ and } k \geq 0,
\end{align}
and (recall that $e_0 = 1$)
\begin{align} \label{eq:e_q_binomial}
	e_k[[n]_q] = q^{\binom{k}{2}} \qbinom{n}{k}_q \quad \text{ for all } n, k \geq 0.
\end{align}
Also (cf. \cite[Corollary~7.21.3]{Stanley-Book-1999})
\begin{align}
	\label{eq:h_q_prspec}
	h_k\left[\frac{1}{1-q}\right] = \frac{1}{(q;q)_k} = \prod_{i=1}^k \frac{1}{1-q^i} \quad \text{ for } k \geq 0,
\end{align}
and
\begin{align}
	\label{eq:e_q_prspec}
	e_k\left[\frac{1}{1-q}\right] = \frac{q^{\binom{k-1}{2}}}{(q;q)_k} = q^{\binom{k-1}{2}} \prod_{i=1}^k \frac{1}{1-q^i} \quad \text{ for } k \geq 0.
\end{align}

\subsection{Useful identities}

In this section we collect some results from the literature that we are going to use later in the text.

The symmetric functions $E_{n,k}$ were introduced in \cite{Garsia-Haglund-qtCatalan-2002} by means of the following expansion:
\begin{align}
	\label{eq:def_Enk}
	e_n \left[ X \frac{1-z}{1-q} \right] = \sum_{k=1}^n \frac{(z;q)_k}{(q;q)_k} E_{n,k}.
\end{align}
Notice that setting $z=q^j$ in \eqref{eq:def_Enk} we get
\begin{align}
	\label{eq:en_q_sum_Enk}
	e_n \left[ X \frac{1-q^j}{1-q} \right] = \sum_{k=1}^n \frac{(q^j;q)_k}{(q;q)_k} E_{n,k} = \sum_{k=1}^n \qbinom{k+j-1}{k}_q E_{n,k} .
\end{align}
In particular, for $j=1$, we get
\begin{align}
	\label{eq:en_sum_Enk}
	e_n = E_{n,1} + E_{n,2} + \cdots +E_{n,n}.
\end{align}
The following identity is \cite[Proposition~2.2]{Garsia-Haglund-qtCatalan-2002}:
\begin{align}
	\label{eq:garsia_haglund_eval}
	\widetilde{H}_\mu[(1-t)(1-q^j)] = (1-q^j) \Pi_\mu h_j[(1-t)B_\mu].
\end{align}
So, using \eqref{eq:Mac_Cauchy} with $Y = [j]_q = \frac{1-q^j}{1-q}$, we get
\begin{align} \label{eq:qn_q_Macexp}
	e_n \left[ X \frac{1-q^j}{1-q} \right] & = \sum_{\mu \vdash n} \frac{\widetilde{H}_\mu[X] \widetilde{H}_\mu [(1-t)(1-q^j)] }{w_\mu} \\
	\notag & = (1-q^j) \sum_{\mu\vdash n} \frac{ \Pi_\mu \widetilde{H}_\mu[X] h_j[(1-t)B_\mu]}{w_\mu}.
\end{align}

For $\mu \vdash n$, Macdonald proved (see \cite{Macdonald-Book-1995}{p.~362} that
\begin{align}
	\label{eq:Mac_hook_coeff_ss}
	\< \widetilde{H}_{\mu}, s_{(n-r,1^r)} \> = e_r[B_{\mu}-1],
\end{align}
so that, since by Pieri rule $e_r h_{n-r} = s_{(n-r,1^r)} + s_{(n-r+1,1^{r-1})}$,
\begin{align}
	\label{eq:Mac_hook_coeff}
	\< \widetilde{H}_{\mu}, e_r h_{n-r} \> = e_r[B_{\mu}].
\end{align}
The following well-known identity is an easy consequence of \eqref{eq:Mac_hook_coeff}.

\begin{lemma}
	\label{lem:Mac_hook_coeff}
	For any symmetric function $f\in \Lambda^{(n)}$,
	\begin{align}
		\label{eq:lem_e_h_Delta}
		\< \Delta_{e_{d}} f, h_n \> = \< f, e_d h_{n-d} \>.
	\end{align}
\end{lemma}

\begin{proof}
	Checking it on the Macdonald basis elements $\widetilde{H}_\mu\in \Lambda^{(n)}$, using \eqref{eq:Mac_hook_coeff}, we get
	\[ \< \Delta_{e_{d}} \widetilde{H}_\mu, h_n \> = e_d[B_\mu] \< \widetilde{H}_\mu, h_n \> = e_d[B_\mu] = \< \widetilde{H}_\mu, e_d h_{n-d} \>. \]
\end{proof}

The following lemma is due to Haglund. 
\begin{lemma}[\cite{Haglund-Schroeder-2004}*{Corollary~2}]
	\label{lem:Haglund}
	For positive integers $d,n$ and any symmetric function $f\in \Lambda^{(n)}$,
	\begin{align} \label{eq:Haglund_Lemma}
		\< \Delta_{e_{d-1}} e_n, f \> = \< \Delta_{\omega f} e_d, h_d \>.
	\end{align}
\end{lemma}

The following theorem is due to Haglund.

\begin{theorem}[\cite{Haglund-Schroeder-2004}*{Theorem~2.11}]
	\label{thm:Haglund_formula}
	For $n,k\in \mathbb{N}$ with $1\leq k\leq n$, \[ \< \nabla E_{n,k},h_n\>=\chi(n=k). \]
	In addition, if $m>0$ and $\lambda \vdash m$,
	\begin{align}
		\< \Delta_{s_\lambda} \nabla E_{n,k}, h_n \> = t^{n-k} \< \Delta_{h_{n-k}} e_m \left[ X \frac{1-q^k}{1-q} \right], s_{\lambda'} \>,
	\end{align}
	or equivalently
	\begin{align}
		\< \Delta_{s_\lambda} \nabla E_{n,k}, h_n\> = t^{n-k} \sum_{\mu \vdash m} \frac{(1-q^k) h_k[(1-t)B_\mu] h_{n-k}[B_\mu] \Pi_{\mu} \widetilde{K}_{\lambda' \mu}}{w_{\mu}} .
	\end{align}
\end{theorem}

We need another theorem of Haglund: the following is essentially \cite{Haglund-Schroeder-2004}*{Theorem~2.5}.
\begin{theorem}
	For $k,n\in \mathbb{N}$ with $1\leq k\leq n$, 
	\begin{align}
		\label{eq:Haglund_nablaEnk}
		\nabla E_{n,k} = t^{n-k}(1-q^k) \mathbf{\Pi} h_k \left[ \frac{X}{1-q} \right] h_{n-k} \left[ \frac{X}{M} \right],
	\end{align}
	where $\mathbf{\Pi}$ is the invertible linear operator defined by
	\begin{equation}
		\mathbf{\Pi} \widetilde{H}_\mu[X] = \Pi_\mu \widetilde{H}_\mu[X] \qquad \text{ for all } \mu.
	\end{equation}
\end{theorem}

The following expansions are well-known, and they can be deduced from Cauchy identities \eqref{eq:Mac_Cauchy}.

\begin{proposition} 
	For $n\in \mathbb{N}$ we have
	\begin{align}
		\label{eq:en_expansion}
		e_n[X] = e_n \left[ \frac{XM}{M} \right] = \sum_{\mu \vdash n} \frac{M B_\mu \Pi_{\mu} \widetilde{H}_\mu[X]}{w_\mu}.
	\end{align}
	Moreover, for all $k\in \mathbb{N}$ with $0\leq k\leq n$, we have
	\begin{align}
		\label{eq:e_h_expansion}
		h_k \left[ \frac{X}{M} \right] e_{n-k} \left[ \frac{X}{M} \right] = \sum_{\mu \vdash n} \frac{e_k[B_\mu] \widetilde{H}_\mu[X]}{w_\mu}.
	\end{align}
\end{proposition}
Finally, we need to recall (part of) a theorem from \cite{DAdderio-VandenWyngaerd-2017}, which turns out to be crucial: set
\begin{align}
\notag F_{n,k}^{(d,\ell)} & \coloneqq \langle \Delta_{h_{\ell}} \Delta_{e_{n-d-\ell}}E_{n-\ell,k},e_{n-\ell}\rangle .
\end{align}
\begin{theorem}[\cite{DAdderio-VandenWyngaerd-2017}*{Theorem~4.6}] \label{thm:dvw}
	For $k,\ell,d\geq 0$, $n\geq k+\ell$ and $n\geq d$, the $F_{n,k}^{(d,\ell)}$ satisfy the following recursion: for $n\geq 1$
	\begin{align}
		F_{n,n}^{(d,\ell)} & =\delta_{\ell,0} q^{\binom{n-d}{2}}\qbinom{n}{d}_q,
	\end{align}
	\begin{align*}
		F_{n,k}^{(d,\ell)}
		& = \chi(n=k+\ell) q^{\binom{k-d}{2}}\qbinom{n-1}{\ell}_q \qbinom{k}{d}_q\\
		& \quad + \sum_{j=0}^{n-k}t^{n-k-j}\sum_{s=0}^{k}\sum_{h=1}^{n-k-j}q^{\binom{s}{2}} \qbinom{k}{s}_q \qbinom{k+j-1}{j}_q \qbinom{s+j-1+h}{h}_q F_{n-k-j,h}^{(d-k+s,\ell-j)}.
	\end{align*}
	with initial conditions
	\begin{align}
		F_{0,k}^{(d,\ell)} & = \delta_{k,0}\delta_{\ell,0}\delta_{d,0}, \qquad F_{n,0}^{(d,\ell)} = \delta_{n,0}\delta_{\ell,0}\delta_{d,0}.
	\end{align}
\end{theorem}

\section{Parallelogram polyominoes}

In this section we reintroduce some notions from \cite{Aval-DAdderio-Dukes-Hicks-LeBorgne-2014}, but adding some decorations into the picture.

In this paper, we will often identify lattice paths with tuples of $0$'s and $1$'s, where $0$'s represent east steps and $1$'s represent north steps.

\begin{definition}
	\label{def:parallelogrampolyominoes}
	A $m \times n$ \textit{parallelogram polyomino} is a pair of $(m+n)$-tuples $(\textbf{r}, \textbf{g})$, where $\textbf{r} = (r_1,\dots,r_{m+n})$, $\textbf{g} = (g_1,\dots,g_{m+n})$ are such that
	
	\begin{enumerate}
		\item  $r_i, g_i \in \{0,1\}$ for all $i$,
		\item if $1 \leq i < m+n$, then $g_1 + \dots + g_i < r_1 + \dots + r_i$,
		\item $g_1 + \dots + g_{m+n} = r_1 + \dots + r_{m+n} = n$.
	\end{enumerate}
\end{definition}

A parallelogram polyomino is actually a pair of paths from $(0,0)$ to $(m,n)$ such that one of the two paths lies always strictly above the other one, never touching it except on the starting and the ending point. We will refer to the top path $\textbf{r}$ as \textit{red path} and to the bottom path $\textbf{g}$ as \textit{green path}. See Figure~\ref{fig:polyo} for an example.

\begin{figure}[!ht]
	\begin{center}
		\begin{tikzpicture}[scale=0.6]
		\draw[step=1.0,gray,opacity=0.6,thin] (0,0) grid (12,7);
		
		\filldraw[yellow, opacity=0.3] (0,0) -- (3,0) -- (3,1) -- (7,1) -- (7,4) -- (10,4) -- (10,5) -- (12,5) -- (12,7) -- (8,7) -- (8,5) -- (5,5) -- (5,4) -- (3,4) -- (3,3) -- (0,3) -- cycle;
		
		\draw[green, line width=3pt] (0,0) -- (3,0) -- (3,1) -- (7,1) -- (7,4) -- (10,4) -- (10,5) -- (12,5) -- (12,7);
		\draw[red, line width=3pt] (0,0) -- (0,3) -- (3,3) -- (3,4) -- (5,4) -- (5,5) -- (8,5) -- (8,7) -- (12,7);
		\end{tikzpicture}
	\end{center}
	
	\caption{A $12 \times 7$ parallelogram polyomino.}
	
	\label{fig:polyo}
	
\end{figure}
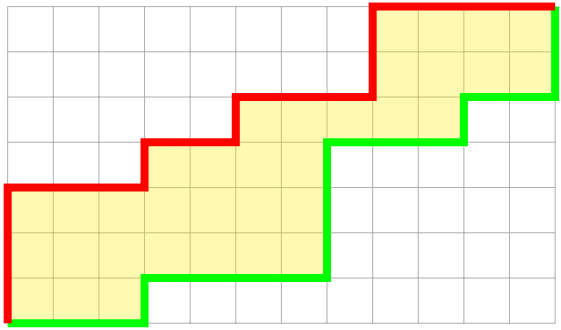

\subsection{The area word}

Parallelogram polyominoes can be coded using their \textit{area word}, which is described in \cite{Aval-DAdderio-Dukes-Hicks-LeBorgne-2014}. The area word can be computed in two equivalent ways.

The first one consists of drawing a diagonal of slope $-1$ from the end of every horizontal green step, and attaching to it the length of that diagonal (i.e. the number of squares it crosses). Then, one puts a dot in every square not crossed by any of those diagonals, and attaches to each vertical red step the number of dots in the corresponding row. Next, one bars the numbers attached to vertical red steps, and finally one reads those numbers following the diagonals of slope $-1$, reading the red label before the green one. See Figure~\ref{fig:areaword} for an example.

\begin{figure}[!ht]
	\begin{center}
		\begin{tikzpicture}[scale=0.6]
			\draw[step=1.0,gray,opacity=0.6,thin] (0,0) grid (12,7);
			
			\filldraw[yellow, opacity=0.3] (0,0) -- (3,0) -- (3,1) -- (7,1) -- (7,4) -- (10,4) -- (10,5) -- (12,5) -- (12,7) -- (8,7) -- (8,5) -- (5,5) -- (5,4) -- (3,4) -- (3,3) -- (0,3) -- cycle;
			
			\draw[black]
			(1,0) -- (0,1)
			(2,0) -- (0,2)
			(3,0) -- (0,3)
			(4,1) -- (2,3)
			(5,1) -- (3,3)
			(6,1) -- (3,4)
			(7,1) -- (4,4)
			(8,4) -- (7,5)
			(9,4) -- (8,5)
			(10,4) -- (8,6)
			(11,5) -- (9,7)
			(12,5) -- (10,7);
			
			\filldraw[fill=black]
			(2.5,1.5) circle (2pt)
			(1.5,2.5) circle (2pt)
			(6.5,2.5) circle (2pt)
			(5.5,3.5) circle (2pt)
			(6.5,3.5) circle (2pt)
			(5.5,4.5) circle (2pt)
			(6.5,4.5) circle (2pt)
			(9.5,5.5) circle (2pt)
			(8.5,6.5) circle (2pt)
			(11.5,6.5) circle (2pt);
			
			\node[below] at (0.5,0) {$1$};
			\node[below] at (1.5,0) {$2$};
			\node[below] at (2.5,0) {$3$};
			\node[below] at (3.5,1) {$2$};
			\node[below] at (4.5,1) {$2$};
			\node[below] at (5.5,1) {$3$};
			\node[below] at (6.5,1) {$3$};
			\node[below] at (7.5,4) {$1$};
			\node[below] at (8.5,4) {$1$};
			\node[below] at (9.5,4) {$2$};
			\node[below] at (10.5,5) {$2$};
			\node[below] at (11.5,5) {$2$};
			
			\node[left] at (0,0.5) {$\bar{0}$};
			\node[left] at (0,1.5) {$\bar{1}$};
			\node[left] at (0,2.5) {$\bar{2}$};
			\node[left] at (3,3.5) {$\bar{2}$};
			\node[left] at (5,4.5) {$\bar{2}$};
			\node[left] at (8,5.5) {$\bar{1}$};
			\node[left] at (8,6.5) {$\bar{2}$};
			
			\draw[green, line width=3pt] (0,0) -- (3,0) -- (3,1) -- (7,1) -- (7,4) -- (10,4) -- (10,5) -- (12,5) -- (12,7);
			\draw[red, line width=3pt] (0,0) -- (0,3) -- (3,3) -- (3,4) -- (5,4) -- (5,5) -- (8,5) -- (8,7) -- (12,7);
		\end{tikzpicture}
	\end{center}
	
	\caption{The construction of the area word for the polyomino in Figure~\ref{fig:polyo}. Its area word is $\bar{0} 1 \bar{1} 2 \bar{2} 3 2 2 \bar{2} 3 3 \bar{2} 1 1 \bar{1} 2 \bar{2} 2 2$.}
	
	\label{fig:areaword}
	
\end{figure}
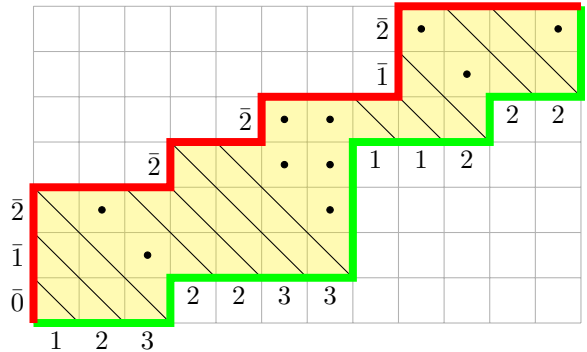

Equivalently, we can build a Dyck path $\textbf{D} = (r_1, 1-g_1, r_2, \dots, r_{m+n}, 1-g_{m+n})$ defined as the interlacing of the two paths $\textbf{r}$ and $\textbf{1-g}$, and then we proceed as follows: starting from level $\bar{0}$, if we read a $1$ then we write down the current level and go up one level in the alphabet $\bar{0} < 1 < \bar{1} < 2 < \dots$ and if we read a $0$, we go down a level without writing down anything.

It is not hard to check that those definitions are equivalent, and that the following holds (see \cite{Aval-DAdderio-Dukes-Hicks-LeBorgne-2014}*{Section~3} for detailed proves and examples).

\begin{theorem}
	For $m, n \geq 1$, there is a bijective correspondence between $m \times n$ parallelogram polyominoes and Dyck words of length $m+n$ in the alphabet $\bar{0} < 1 < \bar{1} < 2 < \dots$ with exactly one $\bar{0}$ (as first letter), exactly $m$ unbarred letters, and exactly $n$ barred letters.
\end{theorem}

We recall here that a \emph{Dyck word} is a word $a_1 a_2 \cdots$ such that if $a_i < a_{i+1}$ then $a_{i+1}$ is the successor of $a_i$ in the alphabet.

\subsection{The statistics $\area$ and $\uarea$}

There is a very natural statistic on parallelogram polyominoes, which is the area.

\begin{definition}
	We define the statistic $\area$ on a parallelogram polyomino as the number of squares between the two paths, or equivalently as the sum of the letters of its area word.
\end{definition}

For example the $\area$ of the polyomino in Figure~\ref{fig:areaword} is $34$.

We want to define a refinement of this statistic for suitably decorated parallelogram polyominoes. The correct spots to decorate are \textit{rises}.

\begin{definition}
	A \textit{rise} of a parallelogram polyomino is a pair of consecutive letters in its area word such that the former is barred, and the latter is its successor in the alphabet (e.g. $\bar{1}2, \bar{2}3, \dots$).
\end{definition}

To decorate a rise, we decorate the barred letter on it. We \emph{do not} decorate the first rise $\bar{0}1$.

\begin{definition}
	We define the statistic $\uarea$ on a parallelogram polyomino with some decorated rises as the sum of the letters of its area word, ignoring the decorated (barred) ones.
\end{definition}

\subsection{The statistics $\bounce$ and $\ubounce$}

We have a second statistic, the $\bounce$. It is computed by drawing the \textit{bounce path}, which is another lattice path going from $(0,0)$ to $(m,n)$.

To draw the bounce path, we draw a single horizontal step, then we pursue the following algorithm: draw vertical steps until the path hits the end of a horizontal red step; then draw horizontal steps until the path hits the end of a vertical green step; repeat until it reaches $(m,n)$.

Now, we attach to each step of the bounce path a letter of the alphabet $\bar{0} < 1 < \bar{1} < 2 < \dots$ starting from $\bar{0}$ and going up a level each time the path changes direction. Let us call \textit{bounce word} the sequence of letters we used. See Figure~\ref{fig:bounce_path} for an example.

\begin{definition}
	We define the statistic $\bounce$ on a parallelogram polyomino as the sum of the letters of its bounce word.
\end{definition}

For example, the polyomino in Figure~\ref{fig:bounce_path} has bounce word $\bar{0}111\bar{1}\bar{1}\bar{1}\bar{1}2\bar{2}\bar{2} 3\bar{3}\bar{3}\bar{3}44\bar{4}\bar{4}$, so its $\bounce$ is $41$.

We want to define a refinement of this statistic for suitably decorated parallelogram polyominoes. The correct spots to decorate are \textit{red peaks}.

\begin{definition}
	A \textit{red peak} of a parallelogram polyomino is a pair of consecutive steps of the red path such that the former is vertical, and the latter is horizontal.
\end{definition}

To decorate a red peak, we decorate the barred letter of the bounce word in the same column of its horizontal step. We do not decorate the leftmost peak. See Figure~\ref{fig:bounce_path} for an example.

\begin{definition}
	We define the statistic $\ubounce$ on a parallelogram polyomino with some decorated red peaks as the sum of the letters of its bounce word, ignoring the decorated (barred) ones.
\end{definition}

For example the $\ubounce$ of the polyomino in Figure~\ref{fig:bounce_path} is $41-1-3=37$.

\begin{figure}[!ht]
	\begin{center}
		\begin{tikzpicture}[scale=0.6]
		\draw[step=1.0,gray,opacity=0.6,thin] (0,0) grid (12,7);
		\filldraw[yellow, opacity=0.3] (0,0) -- (3,0) -- (3,1) -- (5,1) -- (5,3) -- (7,3) -- (7,4) -- (10,4) -- (10,5) -- (12,5) -- (12,7) -- (8,7) -- (8,5) -- (5,5) -- (5,4) -- (3,4) -- (3,3) -- (0,3) -- cycle;
		
		\draw[green, line width=3pt] (0,0) -- (3,0) -- (3,1) -- (5,1) -- (5,3) -- (7,3) -- (7,4) -- (10,4) -- (10,5) -- (12,5) -- (12,7);
		\draw[red, line width=3pt] (0,0) -- (0,3) -- (3,3) -- (3,4) -- (5,4) -- (5,5) -- (8,5) -- (8,7) -- (12,7);
		
		\filldraw[fill=red]
		(3,4) circle (6pt)
		(8,7) circle (6pt);
		
		\draw[blue, line width=1.5pt] (0,0) -- (1,0) -- (1,3) -- (5,3) -- (5,4) -- (7,4) -- (7,5) -- (10,5) -- (10,7) -- (12,7);
		
		\node[blue, above] at (0.5,0) {$\bar{0}$};
		\node[blue, right] at (1,0.5) {$1$};
		\node[blue, right] at (1,1.5) {$1$};
		\node[blue, right] at (1,2.5) {$1$};
		\node[blue, above] at (1.5,3) {$\bar{1}$};
		\node[blue, above] at (2.5,3) {$\bar{1}$};
		\node[red, above] at (3.5,3) {$\bar{1}$};
		\node[blue, above] at (4.5,3) {$\bar{1}$};
		\node[blue, right] at (5,3.5) {$2$};
		\node[blue, above] at (5.5,4) {$\bar{2}$};
		\node[blue, above] at (6.5,4) {$\bar{2}$};
		\node[blue, right] at (7,4.5) {$3$};
		\node[blue, above] at (7.5,5) {$\bar{3}$};
		\node[red, above] at (8.5,5) {$\bar{3}$};
		\node[blue, above] at (9.5,5) {$\bar{3}$};
		\node[blue, right] at (10,5.5) {$4$};
		\node[blue, right] at (10,6.5) {$4$};
		\node[blue, above] at (10.5,7) {$\bar{4}$};
		\node[blue, above] at (11.5,7) {$\bar{4}$};
		
		\end{tikzpicture}
	\end{center}
	
	\caption{A parallelogram polyomino in which the bounce path is shown. It has two decorated red peaks. To compute $\ubounce$, we should sum the values of the blue letters and ignore the red ones.}
	\label{fig:bounce_path}
\end{figure}

\subsection{The statistic $\dinv$}

We have also a third statistic, the $\dinv$, first defined in \cite{Aval-DAdderio-Dukes-Hicks-LeBorgne-2014}. Let us (mysteriously) call \textit{inversion} any pair of letters in the area word of a parallelogram polyomino such that the rightmost is the successor, in the usual alphabet, of the leftmost one.

\begin{definition}
	We define the statistic $\dinv$ on a parallelogram polyomino as the number of its inversions.
\end{definition}

For example, the $\dinv$ of the polyomino in Figure~\ref{fig:areaword} is $32$.

We do not give a decorated version of this statistic. Indeed it can be defined, but the spots to decorate are not as nice as rises or red peaks. Furthermore, we do not need it to state our results.

\subsection{The $\zeta$ map}
\label{sec:zeta}
Before going on, let us fix some notation.
\begin{align}
	&\PP(m,n) && \coloneqq \quad \{ P \mid P \; \text{is a $m \times n$ parallelogram polyomino} \} \\
	&\PP(m \backslash r, n) && \coloneqq \quad \{ P \in \PP(m,n) \mid P \; \text{has $r$ $1$'s in its area word} \} \label{dinvarea1} \\
	&\PP(m,n)^{\star k} && \coloneqq \quad \{ (P, \star_1, \dots, \star_k) \mid P \in \PP(m,n), \; \star_i \; \text{is a decoration on a rise}  \} \label{dinvarea2} \\
	&\PP(m \backslash r, n)^{\star k} && \coloneqq \quad \{ (P, \star_1, \dots, \star_k) \mid P \in \PP(m \backslash r, n), \; \star_i \; \text{is a decoration on a rise}  \} \label{dinvarea3} \\
	&\PP(m, n \backslash s) && \coloneqq \quad \{ P \in \PP(m,n) \mid P \; \text{has $s$ $1$'s in its bounce word} \} \label{areabounce1} \\
	&\PP(m,n)^{\bullet k} && \coloneqq \quad \{ (P, \bullet_1, \dots, \bullet_k) \mid P \in \PP(m,n), \; \bullet_i \; \text{is a decoration on a red peak}  \} \label{areabounce2} \\
	&\PP(m, n \backslash s)^{\bullet k} && \coloneqq \quad \{ (P, \bullet_1, \dots, \bullet_k) \mid P \in \PP(m, n \backslash s), \; \bullet_i \; \text{is a decoration on a red peak}  \} \label{areabounce3}
\end{align}

In \cite[Section~4]{Aval-DAdderio-Dukes-Hicks-LeBorgne-2014}, the authors give a bijection $\zeta \colon \PP(m,n) \rightarrow \PP(n,m)$ swapping $(\area, \bounce)$ and $(\dinv, \area)$. In fact, the same map has a stronger property.

\begin{theorem}
	\label{th:zetamap}
	For $m \geq 1$, $n \geq 1$, $k \geq 0$, and $1 \leq r \leq m$, the bijection $\zeta \colon \PP(m,n) \rightarrow \PP(n,m)$ in \cite[Theorem~4.1]{Aval-DAdderio-Dukes-Hicks-LeBorgne-2014} extends to a bijection $\zeta \colon \PP(m \backslash r, n)^{\star k} \rightarrow \PP(n, m \backslash r)^{\bullet k}$ swapping $(\area, \ubounce)$ and $(\dinv, \uarea)$.
\end{theorem}

\begin{proof}
	First of all, let us recall the definition of the $\zeta$ map. Pick a parallelogram polyomino and draw its bounce path; then, project the labels of the bounce path on both the red and the green path. Now, build the area word of the image as follows: pick the first bounce point on the red path, and write down the $\bar{0}$ and the $1$'s as they appear going downwards along the red path (in this case, the relative order will always be with the $\bar{0}$ first, and all the $1$'s next). Then, go to the first bounce point on the green path, and insert the $\bar{1}$'s after the correct number of $1$'s, in the same relative order in which they appear going downwards to the previous bounce point. If a letter is decorated, keep the decoration. Now, move to the second bounce point on the red path, and repeat. See Figure~\ref{fig:zetamap} for an example.
	
	As proved in \cite[Section~4]{Aval-DAdderio-Dukes-Hicks-LeBorgne-2014}, the result will be the area word of a $n \times m$ parallelogram polyomino. It is also proved that, $\area$ is mapped to $\dinv$, since the squares of the starting parallelogram polyomino correspond to the inversions on the image.
	
	Red peaks are mapped into rises, because when reading the red path top to bottom, one reads the horizontal step first, which corresponds to a barred letter, and the vertical step next, which correspond to the next unbarred letter. Moreover, the decoration is kept on a letter with the same value. This implies that $\ubounce$ is mapped to $\uarea$.
	
	Furthermore, by construction one has that the number of $1$'s in the bounce word is equal to the number of $1$'s in the area word of the image polyomino, since that area word is just an anagram of the bounce word of the starting one.
\end{proof}
	
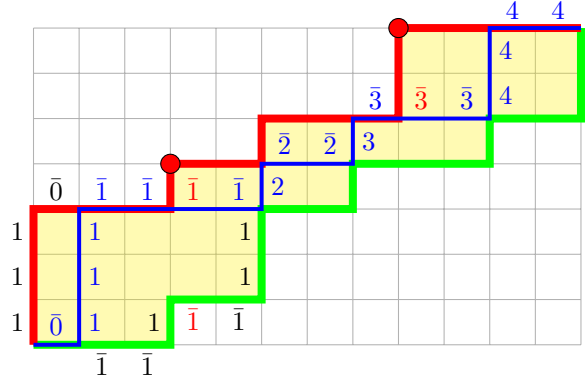
\begin{figure}[!ht]
	\begin{center}
		\begin{tikzpicture}[scale=0.6]
			\draw[step=1.0,gray,opacity=0.6,thin] (0,0) grid (12,7);
			\filldraw[yellow, opacity=0.3] (0,0) -- (3,0) -- (3,1) -- (5,1) -- (5,3) -- (7,3) -- (7,4) -- (10,4) -- (10,5) -- (12,5) -- (12,7) -- (8,7) -- (8,5) -- (5,5) -- (5,4) -- (3,4) -- (3,3) -- (0,3) -- cycle;
			
			\draw[green, line width=3pt] (0,0) -- (3,0) -- (3,1) -- (5,1) -- (5,3) -- (7,3) -- (7,4) -- (10,4) -- (10,5) -- (12,5) -- (12,7);
			\draw[red, line width=3pt] (0,0) -- (0,3) -- (3,3) -- (3,4) -- (5,4) -- (5,5) -- (8,5) -- (8,7) -- (12,7);
			
			\filldraw[fill=red]
			(3,4) circle (6pt)
			(8,7) circle (6pt);
			
			\draw[blue, line width=1.5pt] (0,0) -- (1,0) -- (1,3) -- (5,3) -- (5,4) -- (7,4) -- (7,5) -- (10,5) -- (10,7) -- (12,7);
			
			\node[blue, above] at (0.5,0) {$\bar{0}$};
			\node[blue, right] at (1,0.5) {$1$};
			\node[blue, right] at (1,1.5) {$1$};
			\node[blue, right] at (1,2.5) {$1$};
			\node[blue, above] at (1.5,3) {$\bar{1}$};
			\node[blue, above] at (2.5,3) {$\bar{1}$};
			\node[red, above] at (3.5,3) {$\bar{1}$};
			\node[blue, above] at (4.5,3) {$\bar{1}$};
			\node[blue, right] at (5,3.5) {$2$};
			\node[blue, above] at (5.5,4) {$\bar{2}$};
			\node[blue, above] at (6.5,4) {$\bar{2}$};
			\node[blue, right] at (7,4.5) {$3$};
			\node[blue, above] at (7.5,5) {$\bar{3}$};
			\node[red, above] at (8.5,5) {$\bar{3}$};
			\node[blue, above] at (9.5,5) {$\bar{3}$};
			\node[blue, right] at (10,5.5) {$4$};
			\node[blue, right] at (10,6.5) {$4$};
			\node[blue, above] at (10.5,7) {$\bar{4}$};
			\node[blue, above] at (11.5,7) {$\bar{4}$};
			
			\node[black, left] at (0,0.5) {$1$};
			\node[black, left] at (0,1.5) {$1$};
			\node[black, left] at (0,2.5) {$1$};
			\node[black, above] at (0.5,3) {$\bar{0}$};
			
			\node[black, below] at (1.5,0) {$\bar{1}$};
			\node[black, below] at (2.5,0) {$\bar{1}$};
			\node[black, left] at (3,0.5) {$1$};
			\node[red, below] at (3.5,1) {$\bar{1}$};
			\node[black, below] at (4.5,1) {$\bar{1}$};
			\node[black, left] at (5,1.5) {$1$};
			\node[black, left] at (5,2.5) {$1$};
		
		\end{tikzpicture}
	\end{center}
	
	\caption{The first two steps needed to compute the $\zeta$ map. The final image will be the parallelogram polyomino with area word $\bar{0} 1 1 \bar{1} {\color{red} \bar{1}} 2 \bar{2} \bar{2} 3 \bar{3} {\color{red} \bar{3}} 4 4 \bar{4} \bar{4} \bar{3} 1 \bar{1} \bar{1}$.  }
	\label{fig:zetamap}
\end{figure}

\subsection{$q,t$-enumerators}

We now want to build and compute some $q,t$-enumerators for our sets. Let us define

\[ \PP_{q,t}(m,n) \coloneqq \sum_{P\in \PP(m,n)} q^{\area(P)} t^{\bounce(P)} = \sum_{P\in \PP(m,n)} q^{\dinv(P)} t^{\area(P)} \]

which are proved to be equal in \cite[Section~6]{Aval-DAdderio-Dukes-Hicks-LeBorgne-2014}. We analogously denote with a subscript $q,t$ the $q,t$-enumerator for the other sets defined in Subsection~\ref{sec:zeta}, using the bistatistic $(\dinv, \uarea)$ for \eqref{dinvarea1} - \eqref{dinvarea2} - \eqref{dinvarea3}, and the bistatistic $(\area, \ubounce)$ for \eqref{areabounce1} - \eqref{areabounce2} - \eqref{areabounce3}.

We are now ready to state some recursions.

\begin{theorem}
	\label{th:dinvrecursion}
	For $m \geq 1$, $n \geq 1$, $k \geq 0$, and $1 \leq r \leq m$, the polynomials $\PP_{q,t}(m \backslash r,n)^{\star k}$ satisfy the recursion
	\begin{align*}
	\PP_{q,t}(m \backslash r, n)^{\star k} = t^{m+n-k-1} & \sum_{s=1}^{n-1} q^{r+s} \qbinom{r+s-1}{s}_q \sum_{h=0}^{k} q^{\binom{h}{2}} \qbinom{s}{h}_q \\
	\times & \sum_{u=1}^{m-r} \qbinom{s+u-h-1}{s-1}_q \PP_{q,t}(m-r \, \backslash \, u, \, n-s)^{\star \, k-h}
	\end{align*}
	
	with initial conditions \[ \PP_{q,t}(m \backslash m,n)^{\star 0} = (qt)^{m+n-1} \qbinom{m+n-2}{n-1}_q \] and $\PP_{q,t}(m \backslash r,1)^{\star k} = \delta_{rm}\delta_{k0} (qt)^{m}$.
\end{theorem}

\begin{proof}
	Let $P$ be a polyomino with $k$ decorated rises, $r$ be the number of $1$'s in its area word, $s$ be the number of $\bar{1}$'s, $h$ be the number of $\bar{1}$'s with a decoration, and $u$ be the number of $2$'s. We want to build its area word.
	
	The factor $t^{m+n-k-1}$ takes care of the fact that every non-decorated letter but $\bar{0}$ contribute for at least one unit to the $\uarea$. The factor $q^r$ takes care of the $\dinv$ between the $\bar{0}$ and the $1$'s, and the factor $q^s \qbinom{r+s-1}{s}_q$ takes care of the $\dinv$ between the $1$'s and the $\bar{1}$'s. The factor $q^{\binom{h}{2}} \qbinom{s}{h}_q$ takes care of the $\dinv$ between the decorated $\bar{1}$'s and the $2$'s, where $q^{\binom{h}{2}}$ is needed to take into account the fact that there is at least one $2$ between two decorated $\bar{1}$'s (think of this step as inserting $h$ $2$'s right after some $\bar{1}$, and then decorating those $\bar{1}$'s). Next, the factor $\qbinom{s+u-h-1}{s-1}_q$ takes care of the $\dinv$ between the not decorated $\bar{1}$'s and the $2$'s (think of it as inserting the remaining $u-h$ $2$'s anywhere).
	
	Finally, the contribution to $\dinv$ and $\uarea$ given by the letters greater or equal than $2$ is given by $\PP_{q,t}(m-r \, \backslash \, u, \, n-s)^{\star \, k-h}$. In fact, if we remove from the area word of $P$ all the $1$'s and $\bar{1}$'s, and then we lower by $1$ all the other entries but $\bar{0}$, we get the area word of a polyomino in $\PP(m-r \, \backslash \, u, \, n-s)^{\star \, k-h}$.
	
	The initial conditions are easy to check, and this concludes the proof.
\end{proof}

\begin{theorem}
	\label{th:bouncerecursion}
	
	For $m \geq 1$, $n \geq 1$, $k \geq 0$, and $1 \leq s \leq n$, the polynomials $\PP_{q,t}(m, n \backslash s)^{\bullet k}$ satisfy the recursion
	\begin{align*}
		\PP_{q,t}(m, n \backslash s)^{\bullet k} = t^{m+n-k-1} & \sum_{r=1}^{m-1} q^{r+s} \qbinom{r+s-1}{r}_q \sum_{h=0}^{k} q^{\binom{h}{2}} \qbinom{r}{h}_q \\ \times & \sum_{v=1}^{n-s} \qbinom{r+v-h-1}{r-1}_q \PP_{q,t}(m-r, n-s \, \backslash \, v)^{\bullet \, k-h}
	\end{align*}
	
	with initial conditions \[ \PP_{q,t}(m, n \backslash n)^{\bullet 0} = (qt)^{m+n-1} \qbinom{m+n-2}{m-1}_q \] and $\PP_{q,t}(1, n \backslash s)^{\bullet k} = \delta_{sn}\delta_{k0} (qt)^{n}$.
\end{theorem}

\begin{proof}
	Let $P$ be a polyomino with $k$ decorated red peaks, $s$ be the the number of $1$'s in its bounce word (i.e. the length of the first vertical step of the bounce path), $r$ be the number of $\bar{1}$'s in its bounce word (i.e. the length of the first horizontal step of the bounce path, ignoring the first step), $h$ be the number of decorations in the first $r+1$ columns, and $v$ be the number of $2$'s in the bounce path (i.e. the length of the second vertical step). We want to build its bounce word.
	
	The factor $t^{m+n-k-1}$ takes care of the fact that every non-decorated letter in the bounce word but $\bar{0}$ contribute for at least one unit to the $\ubounce$.
	
	The factor $q^s$ takes care of the $\area$ of the rectangle delimited by the steps of the bounce path labelled by $\bar{0}$ or $1$, and the first $s+1$ steps of the red path (highlighted in lime in Figure~\ref{fig:recursionareabounce}).
	
	The factor $q^r \qbinom{r+s-1}{r}_q$ takes care of the $\area$ of the region delimited by the steps of the bounce path labelled by $1$ or $\bar{1}$, and the first $s+1$ steps of the green path (highlighted in cyan in Figure~\ref{fig:recursionareabounce}).
	
	The factor $q^{\binom{h}{2}} \qbinom{r}{h}_q$ takes care of the $\area$ of the rows corresponding to red peaks in the region delimited by the between the steps of the bounce path labelled by $\bar{1}$ or $2$, and the steps of the red path from the $s+2$-th and the $r+s+v+1$-th (highlighted in pink in Figure~\ref{fig:recursionareabounce}), where $q^{\binom{h}{2}}$ is needed to take into account the fact that if there is a peak of the red path, then there must be a vertical step followed by a horizontal step.
	
	Next, the factor $\qbinom{r+v-h-1}{r-1}_q$ takes care of the $\area$ of the remaining rows in the same region.
	
	Finally, the contribution to $\area$ and $\ubounce$ given by the rest of the polyomino is given by $\PP_{q,t}(m-r, n-s \, \backslash \, v)^{\bullet \, k-h}$. In fact, if we consider the intersection of the original polyomino with the rectangle going from $(r,s)$ to $(m,n)$ (the orange rectangle in Figure~\ref{fig:recursionareabounce}) we get a polyomino in $\PP_{q,t}(m-r, n-s \, \backslash \, v)^{\bullet \, k-h}$.
	
	The initial conditions are easy to check, and this concludes the proof.
\end{proof}

\begin{figure}[!ht]
	\begin{center}
		\begin{tikzpicture}[scale=0.6]
			\draw[step=1.0,gray,opacity=0.6,thin] (0,0) grid (12,7);
			\filldraw[yellow, opacity=0.3] (4,2) -- (5,2) -- (5,3) -- (7,3) -- (7,4) -- (10,4) -- (10,5) -- (12,5) -- (12,7) -- (8,7) -- (8,5) -- (4,5) -- cycle;
			
			\filldraw[lime, opacity=0.3] (0,0) -- (1,0) -- (1,2) -- (0,2) -- cycle;
			\filldraw[cyan, opacity=0.3] (1,0) -- (3,0) -- (3,1) -- (5,1) -- (5,2) -- (1,2) -- cycle;
			\filldraw[pink, opacity=0.3] (2,2) -- (4,2) -- (4,5) -- (2,5) -- cycle;
			
			\draw[green, line width=3pt] (0,0) -- (3,0) -- (3,1) -- (5,1) -- (5,3) -- (7,3) -- (7,4) -- (10,4) -- (10,5) -- (12,5) -- (12,7);
			\draw[red, line width=3pt] (0,0) -- (0,2) -- (2,2) -- (2,5) -- (5,5) -- (8,5) -- (8,7) -- (12,7);
			
			\filldraw[fill=red]
			(2,5) circle (6pt)
			(8,7) circle (6pt);
			
			\draw[blue, line width=1.5pt] (0,0) -- (1,0) -- (1,2) -- (5,2) -- (5,5) -- (10,5) -- (10,7) -- (12,7);
			
			\node[blue, above] at (0.5,0) {$\bar{0}$};
			\node[blue, right] at (1,0.5) {$1$};
			\node[blue, right] at (1,1.5) {$1$};
			\node[blue, above] at (1.5,2) {$\bar{1}$};
			\node[red, above] at (2.5,2) {$\bar{1}$};
			\node[blue, above] at (3.5,2) {$\bar{1}$};
			\node[blue, above] at (4.5,2) {$\bar{1}$};
			\node[blue, right] at (5,2.5) {$2$};
			\node[blue, right] at (5,3.5) {$2$};
			\node[blue, right] at (5,4.5) {$2$};
			\node[blue, above] at (5.5,5) {$\bar{2}$};
			\node[blue, above] at (6.5,5) {$\bar{2}$};
			\node[blue, above] at (7.5,5) {$\bar{2}$};
			\node[red, above] at (8.5,5) {$\bar{2}$};
			\node[blue, above] at (9.5,5) {$\bar{2}$};
			\node[blue, right] at (10,5.5) {$3$};
			\node[blue, right] at (10,6.5) {$3$};
			\node[blue, above] at (10.5,7) {$\bar{3}$};
			\node[blue, above] at (11.5,7) {$\bar{3}$};
			
			\draw[orange, line width = 2pt] (4,2) rectangle (12,7);		
		\end{tikzpicture}
	\end{center}
	
	\caption{The first step of the recursion for $(\area, \ubounce)$.}
	\label{fig:recursionareabounce}
\end{figure}
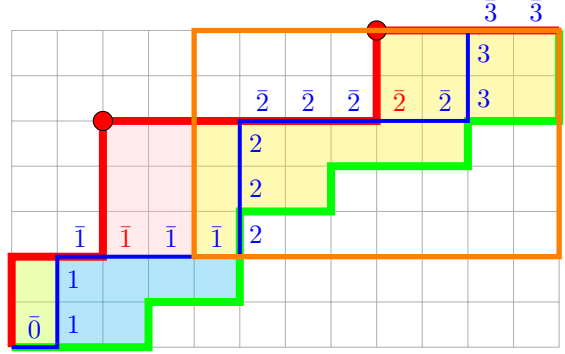

The recursion is the same one given for $(\dinv, \underline{\area})$ switching $m$ and $n$, so we have the following result, which can be also derived from Theorem~\ref{th:zetamap}.

\begin{corollary}
	$\PP_{q,t}(m \backslash r, n)^{\star k} = \PP_{q,t}(n, m \backslash r)^{\bullet k}$.
\end{corollary}

\section{Reduced parallelogram polyominoes}

In this section we discuss a reduced version of the parallelogram polyominoes. Though these objects are strictly related to the original ones, it turns out that some combinatorial results that we will discuss later in the present article look more natural on these reduced siblings.

Condition $2$ in Definition~\ref{def:parallelogrampolyominoes} implies that the red path has to begin with a vertical step and end with a horizontal step, while the green path has to begin with a horizontal step and end with a vertical step. We can adjust the definition in order to remove this restriction, and get \textit{reduced parallelogram polyominoes}.

\begin{definition}
	A $m \times n$ \textit{reduced parallelogram polyomino} is a pair of $(m+n)$-tuples $(\textbf{r}, \textbf{g})$, where $\textbf{r} = (r_1,\dots,r_{m+n})$, $\textbf{g} = (g_1,\dots,g_{m+n})$ are such that
	
	\begin{enumerate}
		\item  $r_i, g_i \in \{0,1\}$ for all $i$,
		\item if $1 \leq i \leq m+n$, then $g_1 + \dots + g_i \leq r_1 + \dots + r_i$,
		\item $g_1 + \dots + g_{m+n} = r_1 + \dots + r_{m+n} = n$.
	\end{enumerate}

\end{definition}

Now, the red path can touch the green path in other points than the extremal two, possibly even with overlapping segments. It still cannot go below it. See Figure~\ref{fig:redbouncepath} for an example.

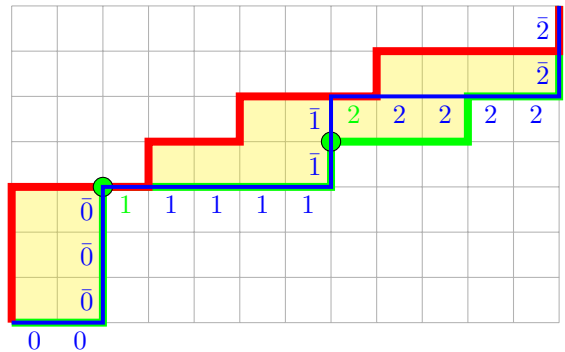
\begin{figure}[!ht]
	\begin{center}
		\begin{tikzpicture}[scale=0.6]
		\draw[step=1.0,gray,opacity=0.6,thin] (0,0) grid (12,7);
		\filldraw[yellow, opacity=0.3] (0,0) -- (2,0) -- (2,3) -- (7,3) -- (7,4) -- (10,4) -- (10,5) -- (12,5) -- (12,6) -- (8,6) -- (8,5) -- (5,5) -- (5,4) -- (3,4) -- (3,3) -- (0,3) -- cycle;
		
		\draw[green, line width=3pt] (0,0) -- (2,0) -- (2,3) -- (7,3) -- (7,4) -- (10,4) -- (10,5) -- (12,5) -- (12,7);
		\draw[red, line width=3pt] (0,0) -- (0,3) -- (3,3) -- (3,4) -- (5,4) -- (5,5) -- (8,5) -- (8,6) -- (12,6) -- (12,7);
		
		\filldraw[fill=green]
		(2,3) circle (6pt)
		(7,4) circle (6pt);
		
		\draw[blue, line width=1.5pt] (0,0) -- (2,0) -- (2,3) -- (7,3) -- (7,5) -- (12,5) -- (12,7);
		
		\node[blue, below] at (0.5,0) {$0$};
		\node[blue, below] at (1.5,0) {$0$};
		\node[blue, left] at (2,0.5) {$\bar{0}$};
		\node[blue, left] at (2,1.5) {$\bar{0}$};
		\node[blue, left] at (2,2.5) {$\bar{0}$};
		\node[green, below] at (2.5,3) {$1$};
		\node[blue, below] at (3.5,3) {$1$};
		\node[blue, below] at (4.5,3) {$1$};
		\node[blue, below] at (5.5,3) {$1$};
		\node[blue, below] at (6.5,3) {$1$};
		\node[blue, left] at (7,3.5) {$\bar{1}$};
		\node[blue, left] at (7,4.5) {$\bar{1}$};
		\node[green, below] at (7.5,5) {$2$};
		\node[blue, below] at (8.5,5) {$2$};
		\node[blue, below] at (9.5,5) {$2$};
		\node[blue, below] at (10.5,5) {$2$};
		\node[blue, below] at (11.5,5) {$2$};
		\node[blue, left] at (12,5.5) {$\bar{2}$};
		\node[blue, left] at (12,6.5) {$\bar{2}$};
		\end{tikzpicture}
	\end{center}
	
	\caption{A reduced polyomino with the bounce path shown, and two decorated green peaks.}
	\label{fig:redbouncepath}
\end{figure}

There is an obvious bijective correspondence $\phi$ between parallelogram polyominoes of size $m \times n$ and reduced parallelogram polyominoes of size $(m-1) \times (n-1)$, consisting in removing the first and the last step of both paths.

We can define the area word of a reduced parallelogram polyomino in the exact same way we did for standard parallelogram polyominoes. For convenience reasons, we artificially add a $0$ at the beginning of the area word we get with the usual algorithm. This way, to get the area word of $\phi(P)$, we only need to remove the $\bar{0}$ at the beginning of the area word of $P$, and then decrease the value of all the letters by $1$. For example, the area word of the reduced polyomino in Figure~\ref{fig:redbouncepath} is $0\bar{0}1\bar{1}2\bar{2}0\bar{0}111\bar{1}211\bar{1}211\bar{0}$.

When working with reduced parallelogram polyominoes, the statistics $\area$, $\bounce$, and $\dinv$ get some sort of normalization.

\subsection{The statistics $\area$ and $\uarea$}

For the area, the definition is natural.

\begin{definition}
	We define the statistic $\area$ on a reduced parallelogram polyomino as the number of squares between the two paths, or equivalently as the sum of the letters of its area word.
\end{definition}
For example, the $\area$ of the reduced polyomino in Figure~\ref{fig:redbouncepath} is $19$.

Once again, we want to define a refinement of this statistic for decorated objects. We still have to decorate rises.

\begin{definition}
	A \textit{rise} of a reduced parallelogram polyomino is a pair of consecutive letters in its area word such that the former is barred, and the latter is its successor in the alphabet (eg. $\bar{1}2, \bar{2}3, \dots$).
\end{definition}

This time, to decorate a rise, we decorate the unbarred letter on it. We can now decorate all the rises, since we do not have the initial (artificial) rise $\bar{0}1$ this time.

\begin{definition}
	We define the statistic $\uarea$ on a parallelogram polyomino with some decorated rises as the sum of the letters of its area word, ignoring the decorated (unbarred) ones.
\end{definition}

For the part concerning $\area$, the difference is very minimal. If $P$ is a $m \times n$ parallelogram polyomino, we have that $\area(P) = \area(\phi(P)) + m + n - 1$, so $\phi$ is just a normalization. The only difference regards the decorated rises, since we have to ignore the unbarred letter instead, in order to get the equality $\uarea(P) = \uarea(\phi(P)) + m + n - 1$.

\subsection{The statistics $\bounce$ and $\ubounce$}

When it comes to the statistic $\bounce$, things change a little bit more.

To draw the bounce path, the algorithm is slightly different. We proceed as follows: start from $(0,0)$ (no artificial first step) and draw horizontal steps until the path hits the \textit{beginning} of a vertical green step; then draw vertical steps until the path hits the \textit{beginning} of a horizontal red step; repeat until it reaches $(m,n)$.

Now, we attach to each step of the bounce path a letter of the alphabet $0 < \bar{0} < 1 < \bar{1} < 2 < \dots$ starting from $0$ and going up a level each time the path changes direction. Once again, call \textit{bounce word} the sequence of letters we used. Notice that we might not use any $0$ (i.e. if both the paths start horizontally). See Figure~\ref{fig:redbouncepath} for an example.

\begin{definition}
	We define the statistic $\bounce$ on a parallelogram polyomino as the sum of the letters of its bounce word.
\end{definition}
So the reduced polyomino in Figure~\ref{fig:redbouncepath} has bounce word $00\bar{0}\bar{0}\bar{0}11111 \bar{1}\bar{1}22222 \bar{2}\bar{2}$, so its $\bounce$ is $21$.

We again want a refinement of this statistic for decorated objects. This time, the correct spots to decorate are \textit{green peaks}.

\begin{definition}
	A \textit{green peak} of a parallelogram polyomino is a pair of consecutive steps of the green path such that the former is vertical, and the latter is horizontal.
\end{definition}

To decorate a green peak, we decorate the unbarred letter of the bounce word in the same column of its horizontal step. As for the area, we are now allowed to decorate any peak.

\begin{definition}
	We define the statistic $\ubounce$ on a reduced parallelogram polyomino with some decorated green peaks as the sum of the letters of its bounce word, ignoring the decorated (unbarred) ones.
\end{definition}

For example the reduced polyomino in Figure~\ref{fig:redbouncepath} has two decorated green peaks, and its $\ubounce$ is $21-1-2=18$.

The difference with the standard case is now more evident, since the starting direction is different, the bouncing algorithm is different, and so are the spots to decorate. Let $\mathsf{s}(P)$ be the reflection along the line $x=y$; it swaps the green and the red path, and maps red peaks into green valleys. Move a decoration on a green valley into the green peak in the same row (there is always exactly one). Then, we have that $\ubounce(\mathsf{s}(P)) = \ubounce(\phi(P)) + m + n - 1$.

In particular, the bounce path of $\mathsf{s}(P)$ can be obtained by starting from $(0,0)$, drawing a horizontal step labelled with $\bar{0}$, then drawing a vertical step labelled by $1$, then reflecting the bounce path of $\phi(P)$ along the line $x=y$ and copying it while increasing all its labels by $1$.

\subsection{The statistic $\dinv$}

When it comes to the $\dinv$, the definition is intrinsically different. Let us (legitimately) call \textit{inversion} any pair of letters in the area word of a parallelogram polyomino such that the leftmost is the successor, in the usual alphabet, of the rightmost one.

\begin{definition}
	We define the statistic $\dinv$ on a reduced parallelogram polyomino as the number of its inversions.
\end{definition}
For example, the $\dinv$ of the reduced polyomino in Figure~\ref{fig:redbouncepath}, whose area word is $0\bar{0}1\bar{1}2\bar{2}0\bar{0}111\bar{1}211\bar{1}211\bar{0}$, is $28$.

This time, we do not in general have that $\dinv(P) = \dinv(\phi(P)) + m + n - 1$ or similar formulas. This substantially different definition of the $\dinv$ statistic will come out to be quite handy later. It is also the main reason to introduce these reduced objects (but not the only one).

\subsection{The $\zeta$ map}

As in the previous case, we have a bijection swapping $(\area, \ubounce)$ and $(\dinv, \uarea)$. This time, it does not swap height and width. Let us fix some notation again.
\begin{align}
	&\RP(m,n) && \coloneqq \quad \{ P \mid P \; \text{is a $m \times n$ reduced parallelogram polyomino} \} \\
	&\RP(m \backslash r, n)^{\star 0} && \coloneqq \quad \{ P \in \RP(m,n) \mid P \; \text{has $r$ $0$'s in its area word} \} \label{mr} \\
	&\RP(m,n)^{\star k} && \coloneqq \quad \{ (P, \star_1, \dots, \star_k) \mid P \in \RP(m,n), \; \star_i \; \text{is a decoration on a rise}  \} \\
	&\RP(m \backslash r, n)^{\star k} && \coloneqq \quad \{ (P, \star_1, \dots, \star_k) \mid P \in \RP(m \backslash r, n)^{\star 0}, \; \star_i \; \text{is a decoration on a rise}  \} \\
	&\RP(m \backslash r, n)^{\bullet 0} && \coloneqq \quad \{ P \in \RP(m,n) \mid P \; \text{has $r-1$ $0$'s in its bounce word} \} \label{ns} \\
	&\RP(m,n)^{\bullet k} && \coloneqq \quad \{ (P, \bullet_1, \dots, \bullet_k) \mid P \in \RP(m,n), \; \bullet_i \; \text{is a decoration on a green peak}  \} \\
	&\RP(m \backslash r, n)^{\bullet k} && \coloneqq \quad \{ (P, \bullet_1, \dots, \bullet_k) \mid P \in \RP(m \backslash r, n)^{\bullet 0}, \; \bullet_i \; \text{is a dec. on a green peak}  \}
\end{align}

Notice that we are including the first, artificial $0$ in \eqref{mr} and that we replaced $r$ with $r-1$ in \eqref{ns}. In particular, $r$ can be equal to $m+1$.

\begin{theorem}
	\label{th:redzetamap}
	For $m \geq 0$, $n \geq 0$, $k \geq 0$, and $1 \leq r \leq m+1$, there is a bijection $\zeta \colon \RP(m \backslash r, n)^{\star k} \rightarrow \RP(m \backslash r, n)^{\bullet k}$ sending the bistatistic $(\area, \ubounce)$ into $(\dinv, \uarea)$.
\end{theorem}

\begin{proof}
	The bijection is essentially the same as the one in Theorem~\ref{th:zetamap}. The only difference is that we have to read the interlacing going \textit{upwards} along the paths (i.e. bottom to top, left to right), and that we have to add the artificial $0$ at the beginning of the area word of the image.
\end{proof}

\subsection{$q,t$-enumerators}

Let us define the $q,t$-enumerators for those new sets analogously as we did for standard decorated parallelogram polyominoes. It is immediate that \[ \PP_{q,t}(m,n) = (qt)^{m+n-1} \cdot \RP_{q,t}(m-1,n-1) \] where the $q,t$-enumerator for $\RP(m-1,n-1)$ is built using the bistatistic $(\area,\bounce)$. This implies that the latter is symmetric in $m$ and $n$; combining this fact with Theorem~\ref{th:redzetamap}, we deduce that we could have equivalently used $(\dinv,\area)$ instead.

\begin{proposition}
	\label{prop:areabounce}
	$\PP_{q,t}(n, m \backslash r)^{\bullet k} = (qt)^{m+n-1} \cdot \RP_{q,t}(m-1 \, \backslash \, r, \, n-1)^{\bullet k}$.
\end{proposition}

\begin{proof}
	Applying $\phi \circ \mathsf{s}$ we get that both statistics decrease by $m+n-1$, and the thesis follows.
\end{proof}

\begin{proposition}
	\label{prop:dinvarea}
	For $m \geq 1$, $n \geq 1$, $k \geq 0$, and $0 \leq r \leq m$, we have \[ \PP_{q,t}(m \backslash r, n)^{\star k} = (qt)^{m+n-1} \cdot \RP_{q,t}(m-1 \, \backslash \, r, \, n-1)^{\star k}. \]
\end{proposition}

\begin{proof}
	Applying $\zeta$ to both terms we get the statement in Proposition~\ref{prop:areabounce}.
\end{proof}

We also have recursions for these sets.

\begin{theorem}
	\label{th:reddinvrecursion}
	
	For $m \geq 0$, $n \geq 0$, $k \geq 0$, and $1 \leq r \leq m+1$, the polynomials $\RP_{q,t}(m \backslash r, n)^{\star k}$ satisfy the recursion
	\begin{align*}
		\RP_{q,t}(m \backslash r, n)^{\star k} = & \sum_{s=1}^{n} t^{m+n+1-r-s-k} \qbinom{r+s-1}{s}_q \sum_{h=0}^{k} q^{\binom{h}{2}} \qbinom{s}{h}_q \\ \times & \sum_{u=1}^{m-r+1} \qbinom{s+u-h-1}{s-1}_q \RP_{q,t}(m-r \, \backslash \, u, \, n-s)^{\star \, k-h}
	\end{align*}
	
	with initial conditions \[ \RP_{q,t}(m \backslash m+1, n)^{\star 0} = \qbinom{m+n}{m}_q \] and $\RP_{q,t}(m \backslash r, 0)^{\star k} = \delta_{k0} \delta_{r(m+1)}$.
\end{theorem}

\begin{theorem}
	\label{th:redbouncerecursion}
	For $m \geq 0$, $n \geq 0$, $k \geq 0$, and $1 \leq s \leq n+1$, the polynomials $\RP_{q,t}(m,n \backslash s)^{\bullet k}$ satisfy the recursion
	\begin{align*}
		\RP_{q,t}(m \backslash r, n)^{\bullet k} = & \sum_{s=1}^{n} t^{m+n+1-r-s-k} \qbinom{r+s-1}{s}_q \sum_{h=0}^{k} q^{\binom{h}{2}} \qbinom{s}{h}_q \\ \times & \sum_{u=1}^{m-r+1} \qbinom{s+u-h-1}{s-1}_q \RP_{q,t}(m-r \, \backslash \, u, \, n-s)^{\bullet \, k-h}
	\end{align*}
	
	with initial conditions \[ \RP_{q,t}(m \backslash m+1, n)^{\bullet 0} = \qbinom{m+n}{m}_q \] and $\RP_{q,t}(m \backslash r, 0)^{\bullet k} = \delta_{k0} \delta_{r(m+1)}$.
\end{theorem}

Both of these can be proved either directly, with the same argument used in Theorem~\ref{th:dinvrecursion} and Theorem~\ref{th:bouncerecursion}, or from the statements of these theorems using Proposition~\ref{prop:dinvarea} and Proposition~\ref{prop:areabounce}.

\section{Two cars parking functions}

Two cars parking functions often appear in problems related to the Delta conjecture. They were first introduced in \cite{Haglund-Haiman-Loehr-Remmel-Ulyanov-2005} (as two-shuffle parking functions). As we are going to see, they are closely related to parallelogram polyominoes.

\begin{definition}
	A \textit{parking function} is a Dyck path in which every vertical step is assigned a label, that must be a positive integer, and these are increasing along columns (see Figure~\ref{fig:parkingfunction}).
\end{definition}

\begin{figure}[!ht]
	\begin{center}
		\begin{tikzpicture}[scale=0.6]
		\draw[step=1.0, gray, opacity=0.6, thin] (0,0) grid (8,8);
		
		\draw[gray, opacity=0.6, thin] (0,0) -- (8,8);
		
		\draw[red, line width=3pt] (0,0) -- (0,1) -- (0,2) -- (0,3) -- (1,3) -- (1,4) -- (2,4) -- (3,4) -- (3,5) -- (4,5) -- (4,6) -- (4,7) -- (4,8) -- (5,8) -- (6,8) -- (7,8) -- (8,8);
		
		\node[red] at (0.5,0.5) {$1$};
		\node[red] at (0.5,1.5) {$2$};
		\node[red] at (0.5,2.5) {$5$};
		\node[red] at (1.5,3.5) {$8$};
		\node[red] at (3.5,4.5) {$7$};
		\node[red] at (4.5,5.5) {$3$};
		\node[red] at (4.5,6.5) {$4$};
		\node[red] at (4.5,7.5) {$6$};
		\end{tikzpicture}
	\end{center}
	
	\caption{A parking functions with labels going from $1$ to $8$.}
	\label{fig:parkingfunction}
\end{figure}
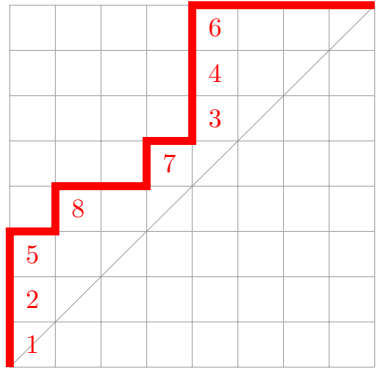

We can assign an area word to a parking function in the usual way, in which the $i$-th letter is the number of full squares between the vertical step in row $i$ and the main diagonal. This allows us to write a parking function as a sequence of dominoes \[ \domino{a_1}{b_1} \; \domino{a_2}{b_2} \; \domino{\cdots}{\cdots} \; \domino{a_n}{b_n}  \] where $a_i$ is the $i$-th label, and $b_i$ is the $i$-th letter of the area word. For example, the parking function in Figure~\ref{fig:parkingfunction} has domino sequence \[ \domino{1}{0} \; \domino{2}{1} \; \domino{5}{2} \; \domino{8}{2} \; \domino{7}{1} \; \domino{3}{1} \; \domino{4}{2} \; \domino{6}{3} \]

Let us recall some general definitions about parking functions.

\begin{definition}
	We define the statistic $\area$ on a parking function as the sum of the letters of its area word.
\end{definition}

For example the parking function in Figure~\ref{fig:parkingfunction} has $\area$ $12$.

We can then decorate \textit{rises}, which are pairs of consecutive vertical steps. Then we can define the statistic $\uarea$ on decorated parking functions as the sum of the letters of the area word, ignoring the ones attached to the topmost step of a decorated rise.

\begin{definition}
	We define the statistic $\dinv$ on a parking function as the number of \textit{inversions}, i.e. the pairs $(i,j)$ with $i < j$ such that $b_i = b_j$ and $a_i < a_j$, or $b_i = b_j + 1$ and $a_i > a_j$.
\end{definition}

A \textit{two cars parking function} is a parking function in which labels can only have value $1$ or $2$ (see Figure~\ref{fig:2cpf}). These objects obviously inherit $\dinv$ and $\area$ statistics from general parking functions.

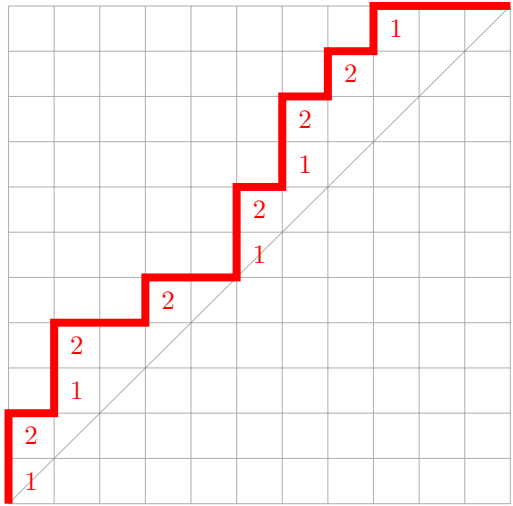
\begin{figure}[!ht]
	\begin{center}
		  \begin{tikzpicture}[scale = 0.6]
				\draw[step=1.0, gray, opacity=0.6, thin] (0,0) grid (11,11);
				\draw[gray, opacity=0.6, thin] (0,0) -- (11,11);
				
				\draw[red, line width=3pt] (0,0) -- (0,1) -- (0,2) -- (1,2) -- (1,3) -- (1,4) -- (2,4) -- (3,4) -- (3,5) -- (4,5) -- (5,5) -- (5,6) -- (5,7) -- (6,7) -- (6,8) -- (6,9) -- (7,9) -- (7,10) -- (8,10) -- (8,11) -- (9,11) -- (10,11) -- (11,11);
				\node[red] at (0.5,0.5) {$1$};
				\node[red] at (0.5,1.5) {$2$};
				\node[red] at (1.5,2.5) {$1$};
				\node[red] at (1.5,3.5) {$2$};
				\node[red] at (3.5,4.5) {$2$};
				\node[red] at (5.5,5.5) {$1$};
				\node[red] at (5.5,6.5) {$2$};
				\node[red] at (6.5,7.5) {$1$};
				\node[red] at (6.5,8.5) {$2$};
				\node[red] at (7.5,9.5) {$2$};
				\node[red] at (8.5,10.5) {$1$};
		\end{tikzpicture}
	\end{center}
	
	\caption{A two cars parking function with five $1$'s and six $2$'s.}
	\label{fig:2cpf}
\end{figure}

\subsection{The statistic $\pmaj$}

We have another statistic on parking functions, the $\pmaj$. It was first introduced in  \cite{Loehr-2005} and \cite{Loehr-Remmel-2004} for those parking functions in which the labels run from $1$ to $n$, where $n$ is the size of the path.

Here we give a slight modification of the algorithm used to compute the $\pmaj$ that allows us to extend the definition for those parking functions with repeated labels.

\begin{definition}
	\label{def:pmaj}
	We define the statistic $\pmaj$ on a parking function $PF$ of size $n$ as follows.
	
	Let $C_1$ be the multiset containing the labels appearing in the first column of $PF$, and let $w_1 \coloneqq \max C_1$. Then, at step $i$, let $C_i$ be the multiset obtained from $C_{i-1}$ by removing $w_{i-1}$ and adding all the labels in the $i$-th column of the $PF$; let $w_i \coloneqq \max \, \{a \in C_i \mid a \leq w_{i-1} \}$ if this last set is non-empty, and $w_i \coloneqq \max \, C_i$ otherwise. Set $w \coloneqq w_1w_2\cdots w_n$. Finally, we define $\pmaj(PF) \coloneqq \mathsf{maj}(w_nw_{n-1} \cdots w_1)$, where $\mathsf{maj}$ is the usual \emph{major index} on words, i.e. $\mathsf{maj}(a_1 a_2 \cdots a_k)$ is the sum of the $i$ such that $a_i>a_{i+1}$, for $i=1,2,\dots k-1$.
\end{definition}

For the two cars parking function in Figure~\ref{fig:2cpf}, we have $w = 22112222111$, hence its $\pmaj$ is $7$. It is known that this $\pmaj$ statistic specializes to the usual $\bounce$ statistic on Dyck paths if the label in row $i$ is $i$ for all $i$.

\subsection{A statistics preserving bijection}

Let us fix some notation.
\begin{align}
	&\PF^2(m,n) && \coloneqq \quad \{ PF \mid PF \; \text{is a two car parking function with $n$ $1$'s and $m$ $2$'s} \} \\
	&\PF^2(m \backslash r, n) && \coloneqq \quad \{ PF \in \PF^2(m,n) \mid PF \; \text{has $r-1$ $2$'s on the main diagonal} \} \\
	&\PF^2(m,n)^{\star k} && \coloneqq \quad \{ (PF, \star_1, \dots, \star_k) \mid PF \in \PF^2(m,n), \; \star_i \; \text{is a decoration on a rise}  \} \\
	&\PF^2(m \backslash r, n)^{\star k} && \coloneqq \quad \{ (PF, \star_1, \dots, \star_k) \mid PF \in \PF^2(m \backslash r, n), \; \star_i \; \text{is a decoration on a rise}  \}
\end{align}

And as usual we add a subscript $q,t$ to denote the relevant $q,t$-enumerators. We have the following result.

\begin{theorem}
	\label{th:polyominoesto2cpf}
	For $m \geq 0$, $n \geq 0$, $k \geq 0$, and $1 \leq r \leq m+1$, there exists a bijection $\psi \colon \RP(m \backslash r, n)^{\star k} \rightarrow \PF^2(m \backslash r, n)^{\star k}$ such that $(\dinv(P),\uarea(P)) = (\dinv(\psi(P)), \uarea(\psi(P)))$ for all $P\in \RP(m \backslash r, n)^{\star k}$.
\end{theorem}

\begin{proof}
	Given the area word of a reduced polyomino, remove the initial artificial $0$. Let $b_i$ be the value of the $i$-th letter, and $a_i$ be $1$ if the $i$-th letter is barred, and $2$ if it is not. This gives the area word of a two cars parking function. Keep the decorations as they are. Both the statistics are trivially preserved (the area word is the same, and the inversions are also the same).
\end{proof}


We can immediately derive the following corollary.

\begin{corollary}
	\label{cor:redpolto2cpf}
	For $m \geq 0$, $n \geq 0$, $k \geq 0$, and $1 \leq r \leq m+1$, we have \[ \RP_{q,t}(m \backslash r, n)^{\star k} = \PF^2_{q,t}(m \backslash r, n)^{\star k}. \]
\end{corollary}

As we have seen, this map preserves the bistatistic $(\dinv, \uarea)$ in a trivial way. The non-trivial result is the following.

\begin{theorem}
	\label{th:pmaj}
	Let $P$ be a reduced polyomino of size $m \times n$. Then $\bounce(P) = \pmaj(\psi(P))$.
\end{theorem}

\begin{proof}
	It is easy to see that the diagonals of slope $-1$ in $P$ correspond to the columns in $\psi(P)$, in the sense that we have a $1$ (resp. a $2$) in column $i$ if and only if we have a vertical red step (resp. horizontal green step) between the diagonals $x+y = i-1$ and $x+y = i$. This is a consequence of the way the area word of a reduced parallelogram polyomino is computed (see Figure~\ref{fig:areaword}).
	
	Hence, while writing down $w$ according to the $\pmaj$ algorithm, we write down $2$'s until we reach the first column with no $2$'s (that could be the actual first one, in which case we write no $2$'s). This means that, on the polyomino, we hit the first diagonal with no horizontal green steps (hence with a vertical green step), so the bounce path is changing direction for the first time. Suppose that we wrote down $r-1$ $2$'s.
	
	Now we want to prove that, if we are writing down a $1$ in $w$, then the bounce path is going upwards. This is trivially true at step $r$ because the bounce path just hit the beginning of a vertical green step. Then, at step $r+i$, if $w_{r+i-1} = 1$, then $w_{r+i} = 1$ if and only if there are at least $i$ $1$'s in the first $r+i$ columns. This means that there are at least $i$ vertical red steps between the diagonals $x+y=0$ and $x+y=r+i$, which implies that the $r+i$-th red step is at least at height $i$, hence the bounce path could not have hit any horizontal red step before its $r+i$-th step, therefore the $r+i$-th step of the bounce path is vertical.
	
	Then, suppose that we wrote down $r-1$ $2$'s, and the next $2$ in $w$ is in position $r+s$ (which will conventionally be $m+n+1$ if there are no more $2$'s). This means that there is no $1$ in column $r+s$, and that there are exactly $s$ $1$'s in the first $r+s-1$ columns. In the polyomino, it means that there are exactly $s$ vertical red steps between the diagonals $x+y = 0$ and $x+y = r+s-1$, that there is a horizontal red step between the diagonals $x+y = r+s-1$ and $x+y = r+s$. Since it must be exactly at height $s$, the bounce path hits it after $r+s-1$ steps (see Figure~\ref{fig:hbounce}).
	
	Now, if $r+s = m+n+1$ both $\bounce(\mathsf{s}(P))$ and $\pmaj(\psi(P))$ are equal to $0$, and we are done. If $r+s \neq m+n+1$, then $w$ reversed has a descent at position $m+n+1-r-s$, and none before that. The bounce word starts with $r-1$ $0$'s and $s$ $\bar{0}$'s, and its remaining $m+n+1-r-s$ letters contribute at least $1$ to the bounce. We can now use a recursive argument on the polyomino delimited by the rectangle $(r,s)$ and $(m,n)$ (we should ignore the step from $(r-1,s)$ to $(r,s)$, since it is forced to be horizontal), and word $w_{r+s+1} \cdots w_{m+n}$. The thesis follows.
\end{proof}

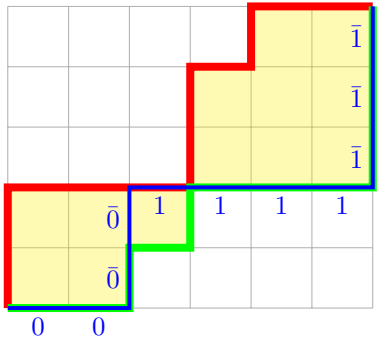
\begin{figure}[!ht]
	\begin{center}
		\begin{tikzpicture}[scale = 0.8]
			\draw[step=1.0, gray, opacity=0.6,thin] (0,0) grid (6,5);
			
			\filldraw[yellow, opacity=0.3] (0,0) -- (1,0) -- (2,0) -- (2,1) -- (3,1) -- (3,2) -- (4,2) -- (5,2) -- (6,2) -- (6,3) -- (6,4) -- (6,5) -- (5,5) -- (4,5) -- (4,4) -- (3,4) -- (3,3) -- (3,2) -- (2,2) -- (1,2) -- (0,2) -- (0,1) -- (0,0);
			
			\draw[red, line width=3pt] (0,0) -- (0,1) -- (0,2) -- (1,2) -- (2,2) -- (3,2) -- (3,3) -- (3,4) -- (4,4) -- (4,5) -- (5,5) -- (6,5);
			
			\draw[green, line width=3pt] (0,0) -- (1,0) -- (2,0) -- (2,1) -- (3,1) -- (3,2) -- (4,2) -- (5,2) -- (6,2) -- (6,3) -- (6,4) -- (6,5);
			
			\draw[blue, line width=1.5pt] (0,0) -- (2,0) -- (2,2) -- (6,2) -- (6,5);
			
			\node[blue, below] at (0.5,0) {$0$};
			\node[blue, below] at (1.5,0) {$0$};
			\node[blue, left] at (2,0.5) {$\bar{0}$};
			\node[blue, left] at (2,1.5) {$\bar{0}$};
			\node[blue, below] at (2.5,2) {$1$};
			\node[blue, below] at (3.5,2) {$1$};
			\node[blue, below] at (4.5,2) {$1$};
			\node[blue, below] at (5.5,2) {$1$};
			\node[blue, left] at (6,2.5) {$\bar{1}$};
			\node[blue, left] at (6,3.5) {$\bar{1}$};
			\node[blue, left] at (6,4.5) {$\bar{1}$};
		\end{tikzpicture}
	\end{center}
	
	\caption{The reduced polyomino corresponding to the two cars parking function in Figure~\ref{fig:2cpf}. Notice that the bounce word $0 0 0 \bar{0} \bar{0} 1 1 1 1 \bar{1} \bar{1} \bar{1}$ follows the same pattern as the pmaj word $22112222111$ of the parking function.}
	\label{fig:hbounce}
\end{figure}

\subsection{$q,t$-enumerators}

In \cite{Wilson-PhD-2015}, the author gives a recursion for rise-decorated two cars parking functions. Of course we have the following.

\begin{theorem}[\cite{Wilson-PhD-2015}*{Proposition~5.3.3.1}]
	\label{th:recursion2cpf}
	For $m \geq 0$, $n \geq 0$, $k \geq 0$, and $1 \leq r \leq m+1$, the polynomials $\PF^2_{q,t}(m \backslash r, n)^{\star k}$ satisfy the recursion
	\begin{align*}
		\PF^2_{q,t}(m \backslash r, n)^{\star k} = & \sum_{s=1}^{n} t^{m+n+1-r-s-k} \qbinom{r+s-1}{s}_q \sum_{h=0}^{k} q^{\binom{h}{2}} \qbinom{s}{h}_q \\ \times & \sum_{u=1}^{m-r+1} \qbinom{s+u-h-1}{s-1}_q \PF^2_{q,t}(m-r \, \backslash \, u, \, n-s)^{\star \, k-h}
	\end{align*}
	
	with initial conditions \[ \PF^2_{q,t}(m \backslash m+1, n)^{\star 0} = \qbinom{m+n}{m}_q \] and $\PF^2_{q,t}(m \backslash r, 0)^{\star k} = \delta_{k0} \delta_{r(m+1)}$.
\end{theorem}

This is exactly the same recursion held in Theorem~\ref{th:reddinvrecursion}, and it can be easily derived from that using Theorem~\ref{th:polyominoesto2cpf}. 

Let us recall the notation used in \cite[Section~5.2.2]{Wilson-PhD-2015} for two cars parking functions (refer to it for the full definitions). We have that the set $RT_{n,m,k}^{(r-1)}(q,t)$ defined there and the set $\PF^2_{q,t}(m \backslash r, n)^{\star k}$ defined in this paper are actually the same.

We now want to explicitly check that the recursion in \cite[Proposition~5.3.3.1]{Wilson-PhD-2015} for $RT_{n,m,k}^{(r-1)}(q,t)$ and the one in Theorem~\ref{th:recursion2cpf} for $\PF^2_{q,t}(m \backslash r, n)^{\star k}$ are equivalent, since they are stated differently.

\begin{proposition}
	$\PF^2_{q,t}(m \backslash r, n)^{\star k}$ and $RT_{n,m,k}^{(r-1)}(q,t)$ satisfy the same recursion.
\end{proposition}

\begin{proof}	
	First of all, by definition we have \[ RT_{a,b,k}^{(s)}(q,t) = \sum_{r=1}^{a} \sum_{i=0}^{k} RT_{a,b,k}^{(r,s,i)}(q,t) \] and, since any parking function with $k$ decorated rises has either $k$ or $k-1$ decorated rises which are not the first one, it holds that \[ RT_{a,b,k}^{(r,s,i)}(q,t) = \overline{RT}_{a,b,k}^{(r,s,i)}(q,t) + t^{-1} \overline{RT}_{a,b,k-1}^{(r,s,i-1)}(q,t) \]
	
	where $t^{-1}$ takes into account the fact that we are adding a decoration to the first rise. We can thus rewrite the recursion as
	\begin{align*}
		\overline{RT}_{a,b,k}^{(r,s,i)}(q,t) & + t^{-1} \overline{RT}_{a,b,k-1}^{(r,s,i-1)}(q,t) \\
		& = q^{\binom{i+1}{2}} t^{a-r+b-s-k} \qbinom{r+s}{r}_q \sum_{h=1}^{b-s} \qbinom{h-1}{i}_q \qbinom{h+r-i-1}{h}_q RT_{a-r,b-s-1,k-i}^{(h-1)}(q,t) \\
		& + q^{\binom{i}{2}} t^{a-r+b-s-k} \qbinom{r+s}{r}_q \sum_{h=1}^{b-s} \qbinom{h-1}{i-1}_q \qbinom{h+r-i}{h}_q RT_{a-r,b-s-1,k-i}^{(h-1)}(q,t)
	\end{align*}
	
	and then as
	\begin{align*}
		RT_{a,b,k}^{(r,s,i)}(q,t) & = q^{\binom{i}{2}} t^{a-r+b-s-k} \qbinom{r+s}{r}_q \\
		& \sum_{h=1}^{b-s} \left( q^{i} \qbinom{h-1}{i}_q \qbinom{h+r-i-1}{h}_q + \qbinom{h-1}{i-1}_q \qbinom{h+r-i}{h}_q \right) RT_{a-r,b-s-1,k-i}^{(h-1)}(q,t).
	\end{align*}
	
	Using \cite[Lemma~2.12]{DAdderio-VandenWyngaerd-2017}, i.e. the identity \[ q^i \qbinom{h-1}{i}_q \qbinom{h+r-i-1}{h}_q + \qbinom{h-1}{i-1}_q \qbinom{h+r-i}{h}_q = \qbinom{r}{i}_q \qbinom{h+r-i-1}{h-i}, \] we have
	\begin{align*}
		RT_{a,b,k}^{(r,s,i)}(q,t) = q^{\binom{i}{2}} t^{a-r+b-s-k+i} \qbinom{r+s}{r}_q \sum_{h=1}^{b-s} \qbinom{r}{i}_q \qbinom{h+r-i-1}{h-i} RT_{a-r,b-s-1,k-i}^{(h-1)}(q,t)
	\end{align*}
	
	and now summing over $r$ and $i$ we get
	\begin{align*}
		RT_{a,b,k}^{(s)}(q,t) = \sum_{r=1}^{a} \sum_{i=0}^{k} q^{\binom{i}{2}} t^{a-r+b-s-k+i} \qbinom{r+s}{r}_q \sum_{h=1}^{b-s} \qbinom{r}{i}_q \qbinom{h+r-i-1}{h-i} RT_{a-r,b-s-1,k-i}^{(h-1)}(q,t).
	\end{align*}
	
	Making the substitutions
	
	\begin{multicols}{3}
		\begin{itemize}
			\item $a \mapsto n$
			\item $b \mapsto m$
			\item $s \mapsto r-1$
			\item $r \mapsto s$
			\item $i \mapsto h$
			\item $h \mapsto u$
		\end{itemize}		
	\end{multicols}
	
	we get
	\begin{align*}
		RT_{n,m,k}^{(r-1)}(q,t) & = \sum_{s=1}^{n} \sum_{h=0}^{k} \sum_{u=1}^{m-r+1} q^{\binom{h}{2}} t^{m+n-r-s-k+1} \\ & \times \qbinom{r+s-1}{s}_q \qbinom{s}{h}_q \qbinom{u+s-h-1}{s-1} RT_{n-s,m-r,k-h}^{(u-1)}(q,t)
	\end{align*}
	
	and now replacing $RT_{n,m,k}^{(r-1)}(q,t)$ with $\PF^2_{q,t}(m \backslash r, n)^{\star k}$ we get the recursion in Theorem~\ref{th:recursion2cpf}, as desired.
\end{proof}

\section{Symmetric functions}

In this section we prove a few identities of symmetric functions.

\subsection{Two identities}

\begin{theorem}
	\label{thm:magic_equality}
	Let $m,n,k\in \mathbb{N}$, $m\geq 0$, $n\geq 0$ and $n\geq k\geq 0$. Then
	\begin{align}
		\label{eq:connection_id}
		\< \Delta_{h_m} e_{n+1}, s_{k+1,1^{n-k}} \> =  \< \Delta_{e_{m+n-k-1}}' e_{m+n}, h_m h_n \>.
	\end{align}
	In particular
	\begin{align}
		\label{eq:symmetry_id}
		\< \Delta_{h_m} e_{n+1}, s_{k+1,1^{n-k}} \> =  \< \Delta_{h_{n}} e_{m+1}, s_{k+1,1^{m-k}} \>.
	\end{align}
\end{theorem}
\begin{proof}
Using \eqref{eq:Mac_hook_coeff}, we have
\begin{align*}
\< \Delta_{h_m} e_{n+1}, h_{k+1}e_{n-k} \>  & = \< \Delta_{h_m e_{n-k}} e_{n+1}, h_{n+1} \> \\
\text{(using \eqref{eq:Haglund_Lemma})}& = \< \Delta_{e_n} e_{m+n-k}, e_m h_{n-k}\> \\
\text{(using \eqref{eq:Mac_hook_coeff})}& = \< \Delta_{e_ne_m} e_{m+n-k},  h_{m+n-k}\> \\
\text{(using \eqref{eq:Haglund_Lemma})} & = \< \Delta_{e_{m+n-k-1}} e_{m+n},  h_{m}h_n\> .
\end{align*}
	Now using this identity, \eqref{eq:deltaprime} and the classical
\begin{align}
h_{k+1}e_{n-k}=s_{k+2,1^{n-k-1}}+s_{k+1,1^{n-k}},
\end{align}
	we get
\begin{align*}
\< \Delta_{h_m} e_{n+1}, s_{k+1,1^{n-k}} \> & = \sum_{j\geq 0}(-1)^{j} \< \Delta_{h_m} e_{n+1}, h_{k+1+j}e_{n-k-j}\>\\
 & = \sum_{j\geq 0}(-1)^{j} \< \Delta_{e_{m+n-k-1-j}} e_{m+n},  h_{m}h_n\> \\
 & =  \< \Delta_{e_{m+n-k-1}}' e_{m+n},  h_{m}h_n\>.
\end{align*}
	This proves \eqref{eq:connection_id}. Now \eqref{eq:symmetry_id} follows immediately from the fact that the right hand side of \eqref{eq:connection_id} is obviously symmetric in $m$ and $n$.
\end{proof}

For the next theorem we need the following lemma from \cite{DAdderio-VandenWyngaerd-2017}.

\begin{lemma}[\cite{DAdderio-VandenWyngaerd-2017}*{Lemma~5.2}]
	For every $n,k\in \mathbb{N}$, with $n> k\geq 1$, $\beta\vdash n$, we have
	\begin{align}
		\label{eq:lemma_sum_cmunu}
		e_{n-k-1}[B_{\beta}-1] B_{\beta} = \sum_{\gamma \subset_k \beta} c_{\beta \gamma}^{(k)} B_{\gamma} T_{\gamma} .
	\end{align}
\end{lemma}

\begin{theorem}
	\label{thm:SF_sum}
	Let $m,n,k\in \mathbb{N}$, $m\geq 0$, $n\geq 0$ and $m\geq k\geq 0$. Then
	\begin{align}
		\sum_{r=1}^{m-k+1} t^{m-k-r+1} \< \Delta_{h_{m-k-r+1}} \Delta_{e_k} e_n \left[ X \frac{1-q^r}{1-q} \right], e_n \> = \< \Delta_{h_m} e_{n+1}, s_{k+1,1^{n-k}} \>.
	\end{align}
\end{theorem}

\begin{proof}
	Using \eqref{eq:qn_q_Macexp}, we have
	\begin{align*}
		& \sum_{r=1}^{m-k+1} t^{m-k-r+1} \< \Delta_{h_{m-k-r+1}} \Delta_{e_k} e_n \left[ X \frac{1-q^r}{1-q} \right], e_n \> \\
		= & \sum_{r=1}^{m-k+1} t^{m-k-r+1} \sum_{\mu\vdash n} (1-q^r) h_r[(1-t)B_\mu] \frac{\Pi_\mu}{w_\mu} h_{m-r-k+1}[B_\mu] e_k[B_\mu] T_\mu \\
		= & \sum_{\mu \vdash n} \left( \sum_{r=1}^{m-k+1} t^{m-k-r+1} (1-q^r) h_r[(1-t)B_\mu] h_{m-r-k+1}[B_\mu] \right) \frac{\Pi_\mu}{w_\mu} e_k[B_\mu] T_\mu \\
		\text{(using \eqref{eq:Haglund_nablaEnk})} = & \sum_{\mu \vdash n} \left( \sum_{r=1}^{m-k+1} \left. \mathbf{\Pi}^{-1} \nabla E_{m-k+1,r}[X] \right|_{X=MB_\mu} \right) \frac{\Pi_\mu}{w_\mu} e_k[B_\mu] T_\mu \\
		\text{(using \eqref{eq:en_sum_Enk})} = & \sum_{\mu \vdash n} \left( \left. \mathbf{\Pi}^{-1} \nabla e_{m-k+1}[X] \right|_{X=MB_\mu} \right) \frac{\Pi_\mu}{w_\mu} e_k[B_\mu] T_\mu \\
		\text{(using \eqref{eq:en_expansion})} = & \sum_{\mu \vdash n} \sum_{\gamma\vdash m-k+1} T_\gamma \frac{MB_\gamma}{w_\gamma} \widetilde{H}_\gamma[MB_\mu] \frac{\Pi_\mu}{w_\mu} e_k[B_\mu] T_\mu \\
		= & \sum_{\gamma\vdash m-k+1} \sum_{\mu \vdash n} T_\mu M \frac{\Pi_\mu}{w_\mu} T_\gamma \frac{B_\gamma}{w_\gamma} e_k \left[ \frac{MB_\mu}{M} \right] \widetilde{H}_\gamma[MB_\mu] \\
		= & \sum_{\gamma \vdash m-k+1} \sum_{\mu \vdash n} T_\mu M \frac{\Pi_\mu}{w_\mu} T_\gamma \frac{B_\gamma}{w_\gamma} \sum_{\alpha \supset_k \gamma} d_{\alpha \gamma}^{(k)}  \widetilde{H}_\alpha[MB_\mu] \\
		\text{(using \eqref{eq:Macdonald_reciprocity})} = & \sum_{\gamma \vdash m-k+1} \sum_{\alpha \supset_k \gamma} M \Pi_{\alpha} \sum_{\mu \vdash n} T_\mu \frac{\widetilde{H}_\mu[MB_\alpha] }{w_\mu} T_\gamma B_\gamma \frac{d_{\alpha \gamma}^{(k)}}{w_\gamma} \\
		\text{(using \eqref{eq:e_h_expansion})} = & \sum_{\gamma \vdash m-k+1} \sum_{\alpha \supset_k \gamma} M \Pi_{\alpha} h_n \left[ \frac{MB_\alpha}{M} \right] T_\gamma B_\gamma \frac{d_{\alpha \gamma}^{(k)}}{w_\gamma} \\
		\text{(using \eqref{eq:rel_cmunu_dmunu})} = & \sum_{\alpha \vdash m+1} \sum_{\gamma \subset_k \alpha} M \Pi_{\alpha} h_n \left[ B_\alpha \right] T_\gamma B_\gamma \frac{c_{\alpha \gamma}^{(k)}}{w_\alpha} \\
		= & \sum_{\alpha \vdash m+1} M \frac{\Pi_{\alpha}}{w_\alpha} h_n \left[ B_\alpha \right] \sum_{\gamma \subset_k \alpha} T_\gamma B_\gamma c_{\alpha \gamma}^{(k)} \\
		\text{(using \eqref{eq:lemma_sum_cmunu})} = & \sum_{\alpha \vdash m+1} M \frac{\Pi_{\alpha}}{w_\alpha} h_n \left[ B_\alpha \right] e_{m-k}[B_\alpha -1] B_\alpha \\
		\text{(using \eqref{eq:Mac_hook_coeff_ss})} = & \< \Delta_{h_n} e_{m+1}, s_{k+1,1^{m-k}} \> \\
		\text{(using \eqref{eq:symmetry_id})} = & \< \Delta_{h_m} e_{n+1}, s_{k+1,1^{n-k}} \>
	\end{align*}
\end{proof}

The following corollary is an immediate consequence of the results in this section.
\begin{corollary} \label{cor:SF_sum}
	Let $m,n,k\in \mathbb{N}$, $m\geq 0$, $n\geq 0$ and $m\geq k\geq 0$. Then
	\begin{align}
		\sum_{r=1}^{m-k+1} t^{m-k-r+1} \< \Delta_{h_{m-k-r+1}} \Delta_{e_k} e_n \left[ X \frac{1-q^r}{1-q} \right], e_n \> & = \< \Delta_{e_{m+n-k-1}}' e_{m+n}, h_m h_n \>.
	\end{align}
\end{corollary}

\section{Main results}

Our polynomials $\RP_{q,t}(m \backslash r, n)^{\star k}$ give a combinatorial interpretation for certain symmetric functions. The following theorem is the main result of the present article.

\begin{theorem} \label{thm:main_qtenumerator}
	For $m \geq 0$, $n \geq 0$, $k \geq 0$, and $1 \leq r \leq m+1$, we have \[ \RP_{q,t}(m \backslash r, n)^{\star k} = t^{m-k-r+1} \< \Delta_{h_{m-k-r+1}} \Delta_{e_k} e_n \left[ X \dfrac{1 - q^r}{1 - q} \right], e_n \> . \]
\end{theorem}

\begin{proof}
	It's enough to show that $\RP_{q,t}(m \backslash r, n)^{\star k}$ and \[ t^{m-k-r+1} \< \Delta_{h_{m-k-r+1}} \Delta_{e_k} e_n \left[ X \dfrac{1 - q^r}{1 - q} \right], e_n \> \] satisfy the same recursion. We already gave the one for $\RP_{q,t}(m \backslash r, n)^{\star k}$.
	
	First of all, we need to look at Theorem~\ref{thm:dvw}. Then we slightly change the statement. If we allow $h$ to be $0$ in the third sum, then since $F_{n,0}^{(d,\ell)} = 0$ unless $n = \ell = d = 0$, the only extra term in the sum is the one with $\ell = j$, $n = k + \ell$, $k = d + s$, and in that case its value agrees with the term \[ \chi(n = k + \ell) \; q^{\binom{k-d}{2}} \qbinom{n-1}{\ell}_q \qbinom{k}{d}_q \] so it follows that we can allow the sum to start from $0$ if we delete that initial term. Then we make the following substitutions:
	
	\begin{multicols}{2}
		\begin{itemize}
			\item $n \mapsto n + m - k - r + 1$
			\item $\ell \mapsto m - k - r + 1$
			\item $d \mapsto n - k$
			\item $k \mapsto s$
		\end{itemize}		
	\end{multicols}
	
	and we get the following recursion (with $s \mapsto h$, $h \mapsto j$, and $j \mapsto u$ as indices).
	\begin{align*}
		F_{n+m-k-r+1,s}^{(n-k,m-k-r+1)} & = \sum_{u=0}^{n+m-k-r+1-s} \sum_{h=0}^k \sum_{j=0}^{n+m-k-r+1-s-u} t^{n+m-k-r+1-s-u} \\ & \times q^{\binom{h}{2}} \qbinom{s}{h}_q \qbinom{s+u-1}{u}_q \qbinom{h+u+j-1}{u}_q F_{n+m-k-r+1-s-u,j}^{(n-k-s+h,m-k-r+1-u)}
	\end{align*}
	
	Now, recall from \eqref{eq:en_q_sum_Enk} that \[ e_n \left[ X \dfrac{1 - q^r}{1 - q} \right] = \sum_{k=1}^{n} \qbinom{k+r-1}{k}_q E_{n,k} \] and the definition \[ F_{n,k}^{(d,\ell)} = \< \Delta_{h_\ell} \Delta_{e_{n-\ell-d}} E_{n-\ell,k}, e_{n-\ell} \> ,\] so, replacing this in the recursion, we get
	\begin{align*}
		\< \Delta_{h_{m-k-r+1}} & \Delta_{e_k} E_{n,s}, e_n \> = \sum_{u=0}^{n+m-k-r+1-s} \sum_{h=0}^k \sum_{j=0}^{n+m-k-r+1-s-u} t^{n+m-k-r+1-s-u} \\ & \times q^{\binom{h}{2}} \qbinom{s}{h}_q \qbinom{s+u-1}{u}_q \qbinom{h+u+j-1}{j}_q \< \Delta_{h_{m-k-r+1-u}} \Delta_{e_{k-h}} E_{n-s,j}, e_{n-s} \>
	\end{align*}
	
	that, since $E_{n,j} = 0$ if $j > n$, and $m-k-r+1 > 0$, we can rewrite as
	\begin{align*}
		\< \Delta_{h_{m-k-r+1}} & \Delta_{e_k} E_{n,s}, e_n \> = \sum_{u=0}^{n+m-k-r+1-s} \sum_{h=0}^k t^{n+m-k-r+1-s-u} \\ & \times q^{\binom{h}{2}} \qbinom{s}{h}_q \qbinom{s+u-1}{u}_q \< \Delta_{h_{m-k-r+1-u}} \Delta_{e_{k-h}} e_{n-s} \left[ X \dfrac{1 - q^{u+h}}{1 - q} \right], e_{n-s} \>
	\end{align*}
	
	Multiplying by $\qbinom{r+s-1}{s}_q$ and summing over $s$ from $1$ to $n$, we get
	\begin{align*}
		\< \Delta_{h_{m-k-r+1}} & \Delta_{e_k} e_n \left[ X \dfrac{1 - q^r}{1 - q} \right], e_n \> = \sum_{s=1}^{n} \sum_{u=0}^{n+m-k-r+1-s} \sum_{h=0}^k t^{n+m-k-r+1-s-u} \\ & \times q^{\binom{h}{2}} \qbinom{r+s-1}{s}_q \qbinom{s}{h}_q \qbinom{s+u-1}{j}_q \< \Delta_{h_{m-k-r+1-u}} \Delta_{e_{k-h}} e_{n-s} \left[ X \dfrac{1 - q^{u+h}}{1 - q} \right], e_{n-s} \>.
	\end{align*}
	
	and now multiplying by $t^{m-k-r+1}$ and noticing that $e_i = 0$ if $i < 0$, we get
	\begin{align*}
		t^{m-k-r+1} & \< \Delta_{h_{m-k-r+1}} \Delta_{e_k} e_n \left[ X \dfrac{1 - q^r}{1 - q} \right], e_n \> \\ & = \sum_{s=1}^{n} \sum_{u=0}^{m-k-r+1} \sum_{h=0}^k t^{m+n-r-s-k+1} q^{\binom{h}{2}} \qbinom{r+s-1}{s}_q \qbinom{s}{h}_q \qbinom{s+u-1}{u}_q \\ & \times t^{m-k-r+1-u} \< \Delta_{h_{m-k-r+1-u}} \Delta_{e_{k-h}} e_{n-s} \left[ X \dfrac{1 - q^{u+h}}{1 - q} \right], e_{n-s} \>.
	\end{align*}
	
	Finally, with the substitution $u \mapsto u - h$, and recalling that $h$ ranges from $0$ to $k$ and one of the $q$-binomials drops to $0$ if $u < h$, we get
	\begin{align*}
		t^{m-k-r+1} & \< \Delta_{h_{m-k-r+1}} \Delta_{e_k} e_n \left[ X \dfrac{1 - q^r}{1 - q} \right], e_n \> \\ & = \sum_{s=1}^{n} \sum_{u=1}^{m-r+1} \sum_{h=0}^k t^{m+n-r-s-k+1} \qbinom{r+s-1}{s}_q q^{\binom{h}{2}} \qbinom{s}{h}_q \qbinom{s+u-h-1}{u-h}_q \\ & \times t^{(m-r)-(k-h)-u+1} \< \Delta_{h_{(m-r)-(k-h)-u+1}} \Delta_{e_{k-h}} e_{n-s} \left[ X \dfrac{1 - q^{u}}{1 - q} \right], e_{n-s} \>.
	\end{align*}
	
	which is exactly the recursion for $\RP_{q,t}(m \backslash r, n)^{\star k}$. The initial conditions are easy to check.	
\end{proof}

\begin{corollary} \label{cor:PF2_qtenum}
	For $m \geq 0$, $n \geq 0$, $k \geq 0$, and $1 \leq r \leq m+1$, we have 
	\[ \PF^2_{q,t}(m \backslash r, n)^{\star k} = t^{m-k-r+1} \< \Delta_{h_{m-k-r+1}} \Delta_{e_k} e_n \left[ X \dfrac{1 - q^r}{1 - q} \right], e_n \> \]
\end{corollary}

The following corollary is an immediate consequence of Theorem~\ref{thm:main_qtenumerator} and Theorem~\ref{thm:magic_equality}. It extends the main results in \cite{Aval-DAdderio-Dukes-Hicks-LeBorgne-2014}.
\begin{corollary}
	For $m \geq 0$, $n \geq 0$, $k \geq 0$, we have
	\begin{align} 
		\RP_{q,t}(m,n)^{\star k} = \< \Delta_{h_m} e_{n+1}, s_{k+1,1^{n-k}} \>
	\end{align}
\end{corollary}

The following corollary is an immediate consequence of Corollary~\ref{cor:PF2_qtenum} and Corollary~\ref{cor:SF_sum}. It settles the ``$hh$'' case of the Delta conjecture in \cite{Haglund-Remmel-Wilson-2015}.
\begin{corollary}
	For $m \geq 0$, $n \geq 0$, $k \geq 0$, we have
	\begin{align} 
		\PF^2_{q,t}(m , n)^{\star k}= \< \Delta_{e_{m+n-k-1}}'e_{m+n},h_{m}h_{n} \> 
	\end{align}
\end{corollary}

\section{Labelled parallelogram polyominoes}

As for Dyck paths, we can add labels to parallelogram polyominoes to get more general objects.

\begin{definition}
	A \textit{labelled parallelogram polyomino} is a standard (i.e. not reduced) parallelogram polyomino in which every vertical step is assigned a label, that must be a positive integer, and these are increasing along columns.
\end{definition}

Let $\LP(m,n)$ be the set of labelled parallelogram polyominoes of size $m \times n$.

\subsection{The statistic $\pmaj$}

We can use a variation of the algorithm in Definition~\ref{def:pmaj} to define a statistic $\pmaj$ on these objects.

\begin{definition}
	We define the statistic $\pmaj$ on a labelled parallelogram polyomino $LP$ of size $m \times n$ as follows.
	
	Let $C_1$ be the multiset containing the labels appearing in the first column of $LP$, and let $w_1 \coloneqq \max C_1$. For $i > 1$, at step $i$, if $g_i = 1$ (i.e. the $i$-th step of the green path is vertical) let $C_i = C_{i-1} \setminus \{w_{i-1}\}$; if $g_i = 0$ (i.e. the $i$-th step is horizontal) let $C_i$ be the multiset obtained from $C_{i-1}$ by replacing $w_{i-1}$ with a $0$, and then adding all the labels in the column of the $LP$ containing the $i$-th green step. Next, let $w_i \coloneqq \max \, \{a \in C_i \mid a \leq w_{i-1} \}$ if it is non-empty, and $w_i \coloneqq \max \, C_i$ otherwise. Set $w \coloneqq w_1w_2\cdots w_{m+n-1}$. Finally, we define $\pmaj(LP) \coloneqq \mathsf{maj}(w_{m+n-1}w_{m+n-2} \cdots w_1)$. See Figure~\ref{fig:pmajpolyo} for an example.
\end{definition}

\begin{figure}[!ht]
	\begin{center}
		\begin{tikzpicture}[scale=0.6]
			\draw[step=1.0,gray,opacity=0.6,thin] (0,0) grid (12,7);
			\filldraw[yellow, opacity=0.3] (0,0) -- (3,0) -- (3,1) -- (5,1) -- (5,3) -- (7,3) -- (7,4) -- (10,4) -- (10,5) -- (12,5) -- (12,7) -- (8,7) -- (8,5) -- (5,5) -- (5,4) -- (3,4) -- (3,3) -- (0,3) -- cycle;
			
			\draw[green, line width=3pt] (0,0) -- (3,0) -- (3,1) -- (5,1) -- (5,3) -- (7,3) -- (7,4) -- (10,4) -- (10,5) -- (12,5) -- (12,7);
			\draw[red, line width=3pt] (0,0) -- (0,3) -- (3,3) -- (3,4) -- (5,4) -- (5,5) -- (8,5) -- (8,7) -- (12,7);
					
			\node[red] at (0.5,0.5) {$1$};			
			\node[red] at (0.5,1.5) {$2$};			
			\node[red] at (0.5,2.5) {$4$};			
			\node[red] at (3.5,3.5) {$7$};			
			\node[red] at (5.5,4.5) {$3$};			
			\node[red] at (8.5,5.5) {$6$};			
			\node[red] at (8.5,6.5) {$8$};	
		\end{tikzpicture}
	\end{center}
	
	\caption{A labelled parallelogram polyomino with $\pmaj$ word $w = 421000073000008600$.}
	\label{fig:pmajpolyo}
\end{figure}

The following theorem says that this new $\pmaj$ statistic generalizes the $\bounce$.

\begin{theorem}
	If $m \geq 1$, $n \geq 1$, and $LP$ is a labelled parallelogram polyomino of size $m \times n$ such that the label in the $i$-th row is equal to $i$ for all $i$, then $\pmaj(LP) = \bounce(LP) - (m+n-1)$. Else, $\pmaj(LP) \leq \bounce(LP) - (m+n-1)$.
	
	Obviously, here the $\bounce$ statistic is computed ignoring the labels.
\end{theorem}

\begin{proof}
	Let's prove the equality first. Let $r$ be the number of consecutive vertical steps at the beginning of the red path. It is clear that this number is both the highest label in the first column, and the number of $1$'s in the bounce word. Hence, we must have $w_i = r-i+1$ for $i \leq r$. Then, using an argument similar to the one in the proof of Theorem~\ref{th:pmaj}, we can deduce that $w_i = 0$ for $r+1 \leq i \leq r+s$, and $w_{r+s+1} \neq 0$, where $s$ is the number of $\bar{1}$'s in the bounce word. As in Theorem~\ref{th:pmaj}, the bounce points correspond to the descents of $w$ reversed, and the additive factor $m+n-1$ takes into account the fact that there are exactly $m+n-1$ nonzero labels in the bounce word, so each one contributes at least $1$. The thesis follows.
	
	If the labels are arranged in any different way, then the highest label in the first column is greater or equal than $r$, hence the first $0$ in $w$ can occur at position $r+1$ or later, and the number of $0$'s we write can only increase if we rearrange the labels. This means that the first descent contributes for $m+n-r-s$ or less to the $\pmaj$. The same argument holds for all the other descents, so in general we have the stated inequality.
\end{proof}

Computer verification suggests that the following conjecture holds.

\begin{conjecture}
	For $m \geq 1$, $n \geq 1$, it holds
	\[ \Delta_{h_{m-1}} e_n = \sum_{P \in \LP(m,n)} q^{\area(P)} t^{\pmaj(P)} x^P. \]
\end{conjecture}

Here by $x^P$ we denote the product of the $x_i$ as $i$ runs over the labels of $P$.

%

\section{Partially labelled Dyck paths}

In \cite{Haglund-Remmel-Wilson-2015}, the authors introduce another set of objects, involved in a generalization of the Delta conjecture. These objects are \textit{partially labelled Dyck paths}.

\begin{definition}
	A \textit{partially labelled Dyck path} is a parking function in which we allow the labels to assume the value $0$, except the first one.
\end{definition}

Here the word \textit{partially} means that the $0$ labels should be ignored when computing the monomials involved in the statements of the conjectures, in a sense that we will clarify later. See Figure~\ref{fig:pldp} for an example.

Notice that, since $0$ is the minimum value a label can assume, and the first label cannot be a $0$, then all the $0$ labels are attached to valleys (i.e. to vertical steps preceded by horizontal steps). We sometimes call these \emph{blank valleys}.

These objects inherit the statistics $\uarea$ and $\dinv$ from parking functions, and using Definition~\ref{def:pmaj}, also the statistic $\pmaj$.

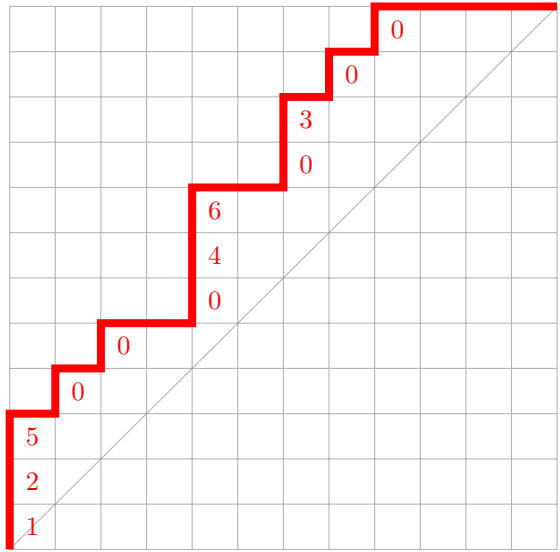
\begin{figure}[!ht]
	\begin{center}
		\begin{tikzpicture}[scale=0.6]
			\draw[step=1.0, gray, opacity=0.6, thin] (0,0) grid (12,12);
			
			\draw[gray, opacity=0.6, thin] (0,0) -- (12,12);
			
			\draw[red, line width=3pt] (0,0) -- (0,1) -- (0,2) -- (0,3) -- (1,3) -- (1,4) -- (2,4) -- (2,5) -- (3,5) -- (4,5) -- (4,6) -- (4,7) -- (4,8) -- (5,8) -- (6,8) -- (6,9) -- (6,10) -- (7,10) -- (7,11) -- (8,11) -- (8,12) -- (9,12) -- (10,12) -- (11,12) -- (12,12);
			
			\node[red] at (0.5,0.5) {$1$};
			\node[red] at (0.5,1.5) {$2$};
			\node[red] at (0.5,2.5) {$5$};
			\node[red] at (1.5,3.5) {$0$};
			\node[red] at (2.5,4.5) {$0$};
			\node[red] at (4.5,5.5) {$0$};
			\node[red] at (4.5,6.5) {$4$};
			\node[red] at (4.5,7.5) {$6$};
			\node[red] at (6.5,8.5) {$0$};
			\node[red] at (6.5,9.5) {$3$};
			\node[red] at (7.5,10.5) {$0$};
			\node[red] at (8.5,11.5) {$0$};
		\end{tikzpicture}
	\end{center}
	
	\caption{A partially labelled Dyck path with $6$ nonzero labels and $6$ zero labels.}
	\label{fig:pldp}
\end{figure}

Let us consider the partially labelled Dyck paths of size $m+n+1$ with $n+1$ nonzero labels, $m$ $0$ labels (that must be valleys), and $k$ decorated rises. Let $\PDP(m,n)^{\star k}$ be this set. In \cite{Haglund-Remmel-Wilson-2015}, the authors give the following conjecture.

\begin{conjecture}[\cite{Haglund-Remmel-Wilson-2015}*{Conjecture~7.4}]
	\label{cong:haglund}
	For $m \geq 0$, $n \geq 0$, $k \geq 0$, it holds
	\[ \Delta_{h_m} \Delta'_{e_{n-k}} e_{n+1} = \sum_{P \in \PDP(m,n)^{\star k}} q^{\dinv(P)} t^{\underline{\area}(P)} x^P. \]
\end{conjecture}

Here by $x^P$ we denote the product of the $x_i$ as $i$ runs over the nonzero labels of $P$. We claim that, in this statement, we can replace $\dinv$ with $\pmaj$, i.e.

\begin{conjecture}
	For $m \geq 0$, $n \geq 0$, $k \geq 0$, it holds
	\[ \Delta_{h_m} \Delta'_{e_{n-k}} e_{n+1} = \sum_{P \in \PDP(m,n)^{\star k}} q^{\underline{\area}(P)} t^{\pmaj(P)} x^P. \]
\end{conjecture}

Other than computer verification, in support of our conjecture, we have some specializations in the cases $m=0$, i.e. in the Delta conjecture case, and $k=n$. If $m=0$, the predicted scalar product with $e_jh_{n-j+1}$ is exactly the combinatorial interpretation with the \emph{second bounce} (denoted $\bounce '$) of \cite[Theorem~6.2]{DAdderio-VandenWyngaerd-2017}. If $k=n$, we have some statistics preserving bijections that allow us to rewrite the conjecture in terms of polyominoes.

\subsection{A statistics preserving bijection}

There is an interesting connection between parallelogram polyominoes and a special subset of partially labelled Dyck paths.

\begin{theorem}
	\label{th:bijection}
	For $m \geq 0$ and $n \geq 0$, there is a bijection $\eta : \PDP(m,n)^{\star n}\to \LP(m+1,n+1)$ such that $\uarea(P)= \area(\eta(P)) - (m+n-1)$ and $\pmaj(P)=\pmaj(\eta(P))$ for all $P\in \PDP(m,n)^{\star n}$.
\end{theorem}

\begin{proof}
	Let $P \in \PDP(m,n)^{\star n}$. Write its area word (the usual one for Dyck paths), colouring in green the numbers corresponding to rows containing a valley, and in red the other ones (for example, the Dyck path in Figure~\ref{fig:pldp} has area word ${\color{red} 0} {\color{red} 1} {\color{red} 2} {\color{green} 2} {\color{green} 2} {\color{green} 1} {\color{red} 2} {\color{red} 3} {\color{green} 2} {\color{red} 3} {\color{green} 3} {\color{green} 3}$).
	
	Now, draw the red path of the polyomino as follows: running over the letters of the area word, draw a vertical step if the letter is red, and a horizontal step if it is green. Whenever you draw a vertical step, attach the label in the corresponding row (there must be one, since vertical steps correspond to rows that do not contain valleys). End with an extra horizontal step.
	
	Next, draw the green path as follows. Start with a horizontal step, then draw a horizontal green step $x+1$ rows below each horizontal red step, where $x$ is the value of the green letter corresponding to that horizontal red step. Then, connect them with vertical steps to get a lattice path from $(0,0)$ to $(m,n)$. See Figure~\ref{fig:bijection} for an example.
	
	It is not hard to see that this is a bijection (its inverse is quite straightforward) and it maps $\uarea$ to the sum of the green letters in the area word, which is exactly the area of the polyomino minus $m+n-1$ (since, in Figure~\ref{fig:bijection}, each letter denotes the number of squares strictly below the one containing it). It also preserves $\pmaj$ by construction, since the algorithm to compute it coincide step by step.
%
	%
\end{proof}

\begin{figure}[!h]
	\begin{center}
		\begin{tikzpicture}[scale=0.6]
			\draw[step=1.0, gray, opacity=0.6,thin] (0,0) grid (7,6);
			
			\filldraw[yellow, opacity=0.3] (0,0) -- (1,0) -- (2,0) -- (3,0) -- (3,1) -- (4,1) -- (4,2) -- (5,2) -- (6,2) -- (7,2) -- (7,3) -- (7,4) -- (7,5) -- (7,6) -- (6,6) -- (5,6) -- (4,6) -- (4,5) -- (3,5) -- (3,4) -- (3,3) -- (2,3) -- (1,3) -- (0,3) -- (0,2) -- (0,1) -- (0,0);
			
			\draw[red, line width=3pt] (0,0) -- (0,1) -- (0,2) -- (0,3) -- (1,3) -- (2,3) -- (3,3) -- (3,4) -- (3,5) -- (4,5) -- (4,6) -- (5,6) -- (6,6) -- (7,6);
			
			\draw[green, line width=3pt] (0,0) -- (1,0) -- (2,0) -- (3,0) -- (3,1) -- (4,1) -- (4,2) -- (5,2) -- (6,2) -- (7,2) -- (7,3) -- (7,4) -- (7,5) -- (7,6);
			
			\node[red] at (0.5,0.5) {$0$};
			\node[red] at (0.5,1.5) {$1$};
			\node[red] at (0.5,2.5) {$2$};
			\node[green] at (1.5,2.5) {$2$};
			\node[green] at (2.5,2.5) {$2$};
			\node[green] at (3.5,2.5) {$1$};
			\node[red] at (3.5,3.5) {$2$};
			\node[red] at (3.5,4.5) {$3$};
			\node[green] at (4.5,4.5) {$2$};
			\node[red] at (4.5,5.5) {$3$};
			\node[green] at (5.5,5.5) {$3$};
			\node[green] at (6.5,5.5) {$3$};
		\end{tikzpicture}
		\begin{tikzpicture}[scale=0.6]
			\draw[step=1.0, gray, opacity=0.6,thin] (0,0) grid (7,6);
			
			\filldraw[yellow, opacity=0.3] (0,0) -- (1,0) -- (2,0) -- (3,0) -- (3,1) -- (4,1) -- (4,2) -- (5,2) -- (6,2) -- (7,2) -- (7,3) -- (7,4) -- (7,5) -- (7,6) -- (6,6) -- (5,6) -- (4,6) -- (4,5) -- (3,5) -- (3,4) -- (3,3) -- (2,3) -- (1,3) -- (0,3) -- (0,2) -- (0,1) -- (0,0);
			
			\draw[red, line width=3pt] (0,0) -- (0,1) -- (0,2) -- (0,3) -- (1,3) -- (2,3) -- (3,3) -- (3,4) -- (3,5) -- (4,5) -- (4,6) -- (5,6) -- (6,6) -- (7,6);
			
			\draw[green, line width=3pt] (0,0) -- (1,0) -- (2,0) -- (3,0) -- (3,1) -- (4,1) -- (4,2) -- (5,2) -- (6,2) -- (7,2) -- (7,3) -- (7,4) -- (7,5) -- (7,6);
			
			\node[white] at (-3,0.5) {};
			\node[red] at (0.5,0.5) {$1$};
			\node[red] at (0.5,1.5) {$2$};
			\node[red] at (0.5,2.5) {$5$};
			\node[red] at (3.5,3.5) {$4$};
			\node[red] at (3.5,4.5) {$6$};
			\node[red] at (4.5,5.5) {$3$};
		\end{tikzpicture}
	\end{center}

	\caption{The image of the partially labelled Dyck path in Figure~\ref{fig:pldp}, together with its new area word (left) or the regular labelling (right).}
	\label{fig:bijection}
\end{figure}
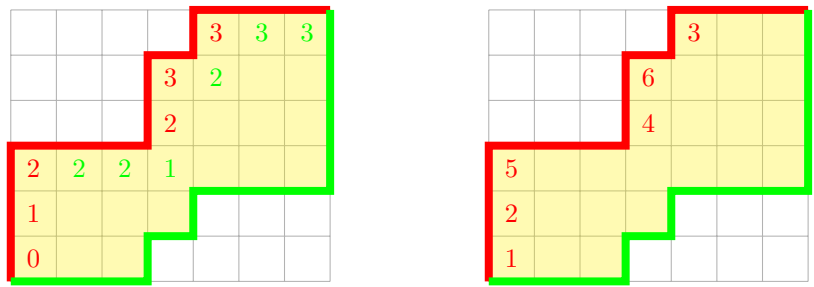

\subsection{A new $\dinv$ statistic on parallelogram polyominoes}

The bijection we just defined gives a new coding for parallelogram polyominoes (using the area word of the corresponding partially labelled Dyck path), together with a new $\dinv$ statistic, which can be seen directly on labelled parallelogram polyominoes as follows.

\begin{definition}
	We define the statistic $\dinv$ on a labelled parallelogram polyomino with new area word $a_1 a_2 \cdots a_{m+n-1}$ as the number of \textit{inversions}, i.e. the pairs $(i,j)$ with $a_i$ red, such that one of the following holds:
	\begin{itemize}
		\item $i > j$, $a_i = a_j$, and either $a_j$ is green or the label attached to the vertical red step corresponding to $a_i$ is greater than the one attached to the step corresponding to $a_j$;
		\item $i < j$, $a_i = a_j + 1$, and either $a_j$ is green or the label attached to the vertical red step corresponding to $a_i$ is greater than the one attached to the step corresponding to $a_j$.
	\end{itemize}
\end{definition}

This definition also applies to unlabelled parallelogram polyominoes assuming that the third condition (the one regarding the label corresponding to $a_j$) always holds, or equivalently that the labels are attached in a suitable canonical ordering (i.e. attaching the labels from $1$ to $n$ starting from the red $0$, then to the $1$'s, next to the $2$'s, and so on, going left to right on each level). This $\dinv$ statistic on labelled polyominoes answers the question in \cite[Equation (8.14)]{Aval-Bergeron-Garsia-2015}.

\subsection{A $\bounce$ statistic on partially labelled Dyck paths}

The bijection in Theorem~\ref{th:bijection} gives also an implicit definition of a $\bounce$ statistic on partially labelled Dyck paths in $\PDP(m,n)^{\star n}$. We can draw a bounce path as usual for Dyck paths, but going \textit{diagonally} in columns containing (blank) valleys, and bouncing whenever the path leaves the main diagonal (see Figure~\ref{fig:plbounce}).

\begin{figure}[!h]
	\begin{center}
		\begin{tikzpicture}[scale=0.4]
			\draw[step=1.0, gray, opacity=0.6, thin] (0,0) grid (12,12);
			
			\draw[gray, opacity=0.6, thin] (0,0) -- (12,12);
			
			\draw[red, line width=3pt] (0,0) -- (0,1) -- (0,2) -- (0,3) -- (1,3) -- (1,4) -- (2,4) -- (2,5) -- (3,5) -- (4,5) -- (4,6) -- (4,7) -- (4,8) -- (5,8) -- (6,8) -- (6,9) -- (6,10) -- (7,10) -- (7,11) -- (8,11) -- (8,12) -- (9,12) -- (10,12) -- (11,12) -- (12,12);

			\draw[blue, line width=1.5pt] (0,0) -- (0,3) -- (1,3) -- (2,4) -- (3,5) -- (4,5) -- (5,6) -- (6,6) -- (7,7) -- (8,8) -- (9,9) -- (9,12) -- (10,12) -- (11,12) -- (12,12);
			
			\node[blue] at (0.5,0.5) {$0$};
			\node[blue] at (0.5,1.5) {$0$};
			\node[blue] at (0.5,2.5) {$0$};
			\node[blue] at (2.5,3.5) {$0$};
			\node[blue] at (3.5,4.5) {$0$};
			\node[blue] at (5.5,5.5) {$0$};
			\node[blue] at (7.5,6.5) {$0$};
			\node[blue] at (8.5,7.5) {$0$};
			\node[blue] at (9.5,8.5) {$0$};
			\node[blue] at (9.5,9.5) {$1$};
			\node[blue] at (9.5,10.5) {$1$};
			\node[blue] at (9.5,11.5) {$1$};
		\end{tikzpicture}
		\begin{tikzpicture}[scale=0.6]
		\draw[step=1.0, gray, opacity=0.6,thin] (0,0) grid (7,6);
		
		\filldraw[yellow, opacity=0.3] (0,0) -- (1,0) -- (2,0) -- (3,0) -- (3,1) -- (4,1) -- (4,2) -- (5,2) -- (6,2) -- (7,2) -- (7,3) -- (7,4) -- (7,5) -- (7,6) -- (6,6) -- (5,6) -- (4,6) -- (4,5) -- (3,5) -- (3,4) -- (3,3) -- (2,3) -- (1,3) -- (0,3) -- (0,2) -- (0,1) -- (0,0);
		
		\draw[red, line width=3pt] (0,0) -- (0,1) -- (0,2) -- (0,3) -- (1,3) -- (2,3) -- (3,3) -- (3,4) -- (3,5) -- (4,5) -- (4,6) -- (5,6) -- (6,6) -- (7,6);
		
		\draw[green, line width=3pt] (0,0) -- (1,0) -- (2,0) -- (3,0) -- (3,1) -- (4,1) -- (4,2) -- (5,2) -- (6,2) -- (7,2) -- (7,3) -- (7,4) -- (7,5) -- (7,6);
		\draw[blue, line width=1.5pt] (0,0) -- (1,0) -- (1,1) -- (1,2) -- (1,3) -- (2,3) -- (3,3) -- (4,3) -- (5,3) -- (6,3) -- (7,3) -- (7,4) -- (7,5) -- (7,6);
		
		\node[white] at (-2,-0.8) {};
		\node[blue, above] at (0.5,0.0) {$\bar{0}$};
		\node[blue, right] at (1.0,0.5) {$1$};
		\node[blue, right] at (1.0,1.5) {$1$};
		\node[blue, right] at (1.0,2.5) {$1$};
		\node[blue, above] at (1.5,3.0) {$\bar{1}$};
		\node[blue, above] at (2.5,3.0) {$\bar{1}$};
		\node[blue, above] at (3.5,3.0) {$\bar{1}$};
		\node[blue, above] at (4.5,3.0) {$\bar{1}$};
		\node[blue, above] at (5.5,3.0) {$\bar{1}$};
		\node[blue, above] at (6.5,3.0) {$\bar{1}$};
		\node[blue, right] at (7.0,3.5) {$2$};
		\node[blue, right] at (7.0,4.5) {$2$};
		\node[blue, right] at (7.0,5.5) {$2$};
		\end{tikzpicture}
	\end{center}
	
	\caption{A partially labelled Dyck path with it bounce path shown (left), and the corresponding polyomino (right). The blue labels on the left are the ones corresponding to vertical or diagonal steps of the bounce path. Vertical steps of the bounce path on the left correspond to vertical ones on the right, while diagonal steps on the left correspond to horizontal ones on the right.}
	\label{fig:plbounce}
\end{figure}
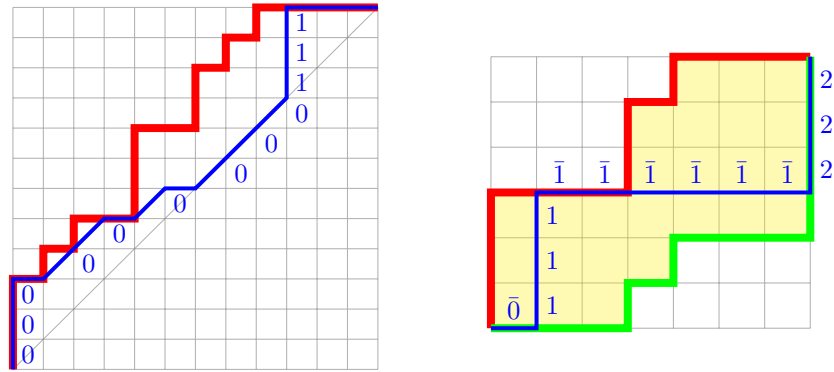

By construction, this $\bounce$ statistic agrees with the $\pmaj$ if the red labels are assigned from $1$ to $n+1$ going bottom to top. In fact, in each run of the pmaj word $w$, we have a number of letters different from $0$ equal to the number of vertical steps in the corresponding part of the bounce path, and a number of $0$'s equal to the number of diagonal steps in the same part: that's because $0$'s correspond to blank valleys, and blank valleys correspond to diagonal steps; any other letter corresponds to a vertical step of the path which is not a blank valley, and these correspond to vertical steps of the bounce path.

\bibliographystyle{plain}
\bibliography{Bibliography}

\begin{bibdiv}
\begin{biblist}

\bib{Aval-Bergeron-Garsia-2015}{article}{
      author={Aval, J.-C.},
      author={Bergeron, F.},
      author={Garsia, A.},
       title={Combinatorics of labelled parallelogram polyominoes},
        date={2015},
        ISSN={0097-3165},
     journal={J. Combin. Theory Ser. A},
      volume={132},
       pages={32\ndash 57},
         url={https://doi.org/10.1016/j.jcta.2014.12.003},
      review={\MR{3311337}},
}

\bib{Aval-DAdderio-Dukes-Hicks-LeBorgne-2014}{article}{
      author={Aval, Jean-Christophe},
      author={D'Adderio, Michele},
      author={Dukes, Mark},
      author={Hicks, Angela},
      author={Le~Borgne, Yvan},
       title={Statistics on parallelogram polyominoes and a {$q,t$}-analogue of
  the {N}arayana numbers},
        date={2014},
        ISSN={0097-3165},
     journal={J. Combin. Theory Ser. A},
      volume={123},
       pages={271\ndash 286},
         url={https://doi.org/10.1016/j.jcta.2013.09.001},
      review={\MR{3157811}},
}

\bib{Bergeron-Garsia-ScienceFiction-1999}{incollection}{
      author={Bergeron, F.},
      author={Garsia, A.~M.},
       title={Science fiction and {M}acdonald's polynomials},
        date={1999},
   booktitle={Algebraic methods and {$q$}-special functions ({M}ontr\'eal,
  {QC}, 1996)},
      series={CRM Proc. Lecture Notes},
      volume={22},
   publisher={Amer. Math. Soc., Providence, RI},
       pages={1\ndash 52},
      review={\MR{1726826}},
}

\bib{Bergeron-Garsia-Haiman-Tesler-Positivity-1999}{article}{
      author={Bergeron, F.},
      author={Garsia, A.~M.},
      author={Haiman, M.},
      author={Tesler, G.},
       title={Identities and positivity conjectures for some remarkable
  operators in the theory of symmetric functions},
        date={1999},
        ISSN={1073-2772},
     journal={Methods Appl. Anal.},
      volume={6},
      number={3},
       pages={363\ndash 420},
         url={https://doi.org/10.4310/MAA.1999.v6.n3.a7},
        note={Dedicated to Richard A. Askey on the occasion of his 65th
  birthday, Part III},
      review={\MR{1803316}},
}

\bib{Bergeron-Haiman-2013}{article}{
      author={Bergeron, Fran\c{c}ois},
      author={Haiman, Mark},
       title={Tableaux formulas for {M}acdonald polynomials},
        date={2013},
        ISSN={0218-1967},
     journal={Internat. J. Algebra Comput.},
      volume={23},
      number={4},
       pages={833\ndash 852},
         url={https://doi.org/10.1142/S0218196713400122},
      review={\MR{3078059}},
}

\bib{Carlsson-Mellit-ShuffleConj-2015}{article}{
      author={{Carlsson}, E.},
      author={{Mellit}, A.},
       title={{A proof of the shuffle conjecture}},
        date={2015-08},
     journal={ArXiv e-prints},
      eprint={1508.06239},
}

\bib{DAdderio-VandenWyngaerd-2017}{article}{
      author={{D'Adderio}, M.},
      author={{Vanden Wyngaerd}, A.},
       title={{Decorated Dyck paths, the Delta conjecture, and a new
  q,t-square}},
        date={2017-09},
     journal={ArXiv e-prints},
      eprint={1709.08736},
}

\bib{Garsia-Haglund-Remmel-Yoo-2017}{article}{
      author={{Garsia}, A.},
      author={{Haglund}, J.},
      author={{Remmel}, J.~B.},
      author={{Yoo}, M.},
       title={{A proof of the Delta Conjecture when $q=0$}},
        date={2017-10},
     journal={ArXiv e-prints},
      eprint={1710.07078},
}

\bib{Garsia-Haglund-qtCatalan-2002}{article}{
      author={Garsia, A.~M.},
      author={Haglund, J.},
       title={A proof of the {$q,t$}-{C}atalan positivity conjecture},
        date={2002},
        ISSN={0012-365X},
     journal={Discrete Math.},
      volume={256},
      number={3},
       pages={677\ndash 717},
         url={https://doi.org/10.1016/S0012-365X(02)00343-6},
        note={LaCIM 2000 Conference on Combinatorics, Computer Science and
  Applications (Montreal, QC)},
      review={\MR{1935784}},
}

\bib{Garsia-Haiman-qLagrange-1996}{article}{
      author={Garsia, A.~M.},
      author={Haiman, M.},
       title={A remarkable {$q,t$}-{C}atalan sequence and {$q$}-{L}agrange
  inversion},
        date={1996},
        ISSN={0925-9899},
     journal={J. Algebraic Combin.},
      volume={5},
      number={3},
       pages={191\ndash 244},
         url={https://doi.org/10.1023/A:1022476211638},
      review={\MR{1394305}},
}

\bib{Garsia-Haiman-Tesler-Explicit-1999}{article}{
      author={Garsia, A.~M.},
      author={Haiman, M.},
      author={Tesler, G.},
       title={Explicit plethystic formulas for {M}acdonald {$q,t$}-{K}ostka
  coefficients},
        date={1999},
        ISSN={1286-4889},
     journal={S\'em. Lothar. Combin.},
      volume={42},
       pages={Art. B42m, 45},
        note={The Andrews Festschrift (Maratea, 1998)},
      review={\MR{1701592}},
}

\bib{Garsia-Haglund-Xin-Zabrocki-Pieri-2016}{incollection}{
      author={Garsia, Adriano},
      author={Haglund, Jim},
      author={Xin, Guoce},
      author={Zabrocki, Mike},
       title={Some new applications of the {S}tanley-{M}acdonald {P}ieri
  rules},
        date={2016},
   booktitle={The mathematical legacy of {R}ichard {P}. {S}tanley},
   publisher={Amer. Math. Soc., Providence, RI},
       pages={141\ndash 168},
      review={\MR{3617221}},
}

\bib{Garsia-Haiman-PNAS-1993}{article}{
      author={Garsia, Adriano~M.},
      author={Haiman, Mark},
       title={A graded representation model for {M}acdonald's polynomials},
        date={1993},
        ISSN={0027-8424},
     journal={Proc. Nat. Acad. Sci. U.S.A.},
      volume={90},
      number={8},
       pages={3607\ndash 3610},
         url={https://doi.org/10.1073/pnas.90.8.3607},
      review={\MR{1214091}},
}

\bib{Haglund-Schroeder-2004}{article}{
      author={Haglund, J.},
       title={A proof of the {$q,t$}-{S}chr\"oder conjecture},
        date={2004},
        ISSN={1073-7928},
     journal={Int. Math. Res. Not.},
      number={11},
       pages={525\ndash 560},
         url={https://doi.org/10.1155/S1073792804132509},
      review={\MR{2038776}},
}

\bib{HHLRU-2005}{article}{
      author={Haglund, J.},
      author={Haiman, M.},
      author={Loehr, N.},
      author={Remmel, J.~B.},
      author={Ulyanov, A.},
       title={A combinatorial formula for the character of the diagonal
  coinvariants},
        date={2005},
        ISSN={0012-7094},
     journal={Duke Math. J.},
      volume={126},
      number={2},
       pages={195\ndash 232},
         url={https://doi.org/10.1215/S0012-7094-04-12621-1},
      review={\MR{2115257}},
}

\bib{Haglund-Haiman-Loehr-Remmel-Ulyanov-2005}{article}{
      author={Haglund, J.},
      author={Haiman, M.},
      author={Loehr, N.},
      author={Remmel, J.~B.},
      author={Ulyanov, A.},
       title={A combinatorial formula for the character of the diagonal
  coinvariants},
        date={2005},
        ISSN={0012-7094},
     journal={Duke Math. J.},
      volume={126},
      number={2},
       pages={195\ndash 232},
         url={https://doi.org/10.1215/S0012-7094-04-12621-1},
      review={\MR{2115257}},
}

\bib{Haglund-Morse-Zabrocki-2012}{article}{
      author={Haglund, J.},
      author={Morse, J.},
      author={Zabrocki, M.},
       title={A compositional shuffle conjecture specifying touch points of the
  {D}yck path},
        date={2012},
        ISSN={0008-414X},
     journal={Canad. J. Math.},
      volume={64},
      number={4},
       pages={822\ndash 844},
         url={https://doi.org/10.4153/CJM-2011-078-4},
      review={\MR{2957232}},
}

\bib{Haglund-Remmel-Wilson-2015}{article}{
      author={{Haglund}, J.},
      author={{Remmel}, J.},
      author={{Wilson}, A.~T.},
       title={{The Delta Conjecture}},
        date={2015-09},
     journal={ArXiv e-prints},
      eprint={1509.07058},
}

\bib{Haglund-Book-2008}{book}{
      author={Haglund, James},
       title={The {$q$},{$t$}-{C}atalan numbers and the space of diagonal
  harmonics},
      series={University Lecture Series},
   publisher={American Mathematical Society, Providence, RI},
        date={2008},
      volume={41},
        ISBN={978-0-8218-4411-3; 0-8218-4411-3},
        note={With an appendix on the combinatorics of Macdonald polynomials},
      review={\MR{2371044}},
}

\bib{Haiman-nfactorial-2001}{article}{
      author={Haiman, Mark},
       title={Hilbert schemes, polygraphs and the {M}acdonald positivity
  conjecture},
        date={2001},
        ISSN={0894-0347},
     journal={J. Amer. Math. Soc.},
      volume={14},
      number={4},
       pages={941\ndash 1006},
         url={https://doi.org/10.1090/S0894-0347-01-00373-3},
      review={\MR{1839919}},
}

\bib{Haiman-Vanishing-2002}{article}{
      author={Haiman, Mark},
       title={Vanishing theorems and character formulas for the {H}ilbert
  scheme of points in the plane},
        date={2002},
        ISSN={0020-9910},
     journal={Invent. Math.},
      volume={149},
      number={2},
       pages={371\ndash 407},
         url={https://doi.org/10.1007/s002220200219},
      review={\MR{1918676}},
}

\bib{Loehr-2005}{article}{
      author={Loehr, Nicholas~A.},
       title={Combinatorics of {$q$}, {$t$}-parking functions},
        date={2005},
        ISSN={0196-8858},
     journal={Adv. in Appl. Math.},
      volume={34},
      number={2},
       pages={408\ndash 425},
         url={https://doi.org/10.1016/j.aam.2004.08.002},
      review={\MR{2110560}},
}

\bib{Loehr-Remmel-2004}{article}{
      author={Loehr, Nicholas~A.},
      author={Remmel, Jeffrey~B.},
       title={Conjectured combinatorial models for the {H}ilbert series of
  generalized diagonal harmonics modules},
        date={2004},
        ISSN={1077-8926},
     journal={Electron. J. Combin.},
      volume={11},
      number={1},
       pages={Research Paper 68, 64},
         url={http://www.combinatorics.org/Volume_11/Abstracts/v11i1r68.html},
      review={\MR{2097334}},
}

\bib{Macdonald-Book-1995}{book}{
      author={Macdonald, I.~G.},
       title={Symmetric functions and {H}all polynomials},
     edition={Second},
      series={Oxford Mathematical Monographs},
   publisher={The Clarendon Press, Oxford University Press, New York},
        date={1995},
        ISBN={0-19-853489-2},
        note={With contributions by A. Zelevinsky, Oxford Science
  Publications},
      review={\MR{1354144}},
}

\bib{Romero-Deltaq1-2017}{article}{
      author={Romero, Marino},
       title={The delta conjecture at {$q = 1$}},
        date={2017},
        ISSN={0002-9947},
     journal={Trans. Amer. Math. Soc.},
      volume={369},
      number={10},
       pages={7509\ndash 7530},
         url={https://doi.org/10.1090/tran/7140},
      review={\MR{3683116}},
}

\bib{Stanley-Book-1999}{book}{
      author={Stanley, Richard~P.},
       title={Enumerative combinatorics. {V}ol. 2},
      series={Cambridge Studies in Advanced Mathematics},
   publisher={Cambridge University Press, Cambridge},
        date={1999},
      volume={62},
        ISBN={0-521-56069-1; 0-521-78987-7},
         url={https://doi.org/10.1017/CBO9780511609589},
        note={With a foreword by Gian-Carlo Rota and appendix 1 by Sergey
  Fomin},
      review={\MR{1676282}},
}

\bib{Wilson-PhD-2015}{book}{
      author={Wilson, Andrew~Timothy},
       title={Generalized {S}huffle {C}onjectures for the {G}arsia-{H}aiman
  {D}elta {O}perator},
   publisher={ProQuest LLC, Ann Arbor, MI},
        date={2015},
        ISBN={978-1321-85316-2},
  url={http://gateway.proquest.com/openurl?url_ver=Z39.88-2004&rft_val_fmt=info:ofi/fmt:kev:mtx:dissertation&res_dat=xri:pqm&rft_dat=xri:pqdiss:3709671},
        note={Thesis (Ph.D.)--University of California, San Diego},
      review={\MR{3407349}},
}

\bib{Zabrocki-4Catalan-2016}{article}{
      author={{Zabrocki}, M.},
       title={{A proof of the $4$-variable Catalan polynomial of the Delta
  conjecture}},
        date={2016-09},
     journal={ArXiv e-prints},
      eprint={1609.03497},
}

\end{biblist}
\end{bibdiv}

\end{document}